\documentclass[a4paper,12pt]{article}

\usepackage{amsthm,amsmath,amssymb,nicefrac,bbm,epsfig,enumerate,geometry}
\usepackage{xcolor}
\usepackage[colorlinks,linkcolor=red,citecolor=green]{hyperref}
\usepackage[final,notref,notcite]{showkeys}
\usepackage[pagewise]{lineno}[4.41]
\usepackage[capitalise,nameinlink,noabbrev]{cleveref}
\newtheorem{lemma}{Lemma}[section]
\newtheorem{remark}[lemma]{Remark}
\newtheorem{proposition}[lemma]{Proposition}
\newtheorem{theorem}[lemma]{Theorem}

\newtheorem{corollary}[lemma]{Corollary}

\newcommand*\Laplace{\mathop{}\!\mathbin\bigtriangleup}

\newcommand{\1}{\mathbbm{1}}
\providecommand{\F}{{\ensuremath{\mathbbm{F}}}}
\providecommand{\N}{{\ensuremath{\mathbbm{N}}}}

\providecommand{\R}{{\ensuremath{\mathbbm{R}}}}

\providecommand{\E}{{\ensuremath{\mathbb{E}}}}
\renewcommand{\P}{{\ensuremath{\mathbb{P}}}}
\newcommand{\lipconstY}{L_\mathfrak{y}}
\newcommand{\lipconstZ}{L_\mathfrak{z}}

\newcommand{\sfF}{\mathsf{F}}
\allowdisplaybreaks

\author{Martin Hutzenthaler\thanks{Faculty of Mathematics, University of Duisburg-Essen, Essen, Germany;\newline\hspace*{.5cm} e-mail: \texttt{martin.hutzenthaler}\textcircled{\texttt{a}}\texttt{uni-due.de}}
\and
Thomas Kruse\thanks{Institute of Mathematics, University of Gie{\ss}en, Gie{\ss}en, Germany;\newline\hspace*{.5cm} e-mail: \texttt{thomas.kruse}\textcircled{\texttt{a}}\texttt{math.uni-giessen.de}}
\and 
Tuan Anh Nguyen\thanks{Faculty of Mathematics, University of Duisburg-Essen, Essen, Germany; \newline\hspace*{.5cm} e-mail: \texttt{tuan.nguyen}\textcircled{\texttt{a}}\texttt{uni-due.de}}
}

\begin{document}\sloppy

\title{On the speed of convergence of Picard iterations\\
of backward stochastic differential equations}
\maketitle
{
\let\thefootnote\relax\footnotetext{\emph{Key words and phrases:} backward stochastic differential equation, Picard iteration, a priori estimate, semilinear parabolic partial differential equation}
\let\thefootnote\relax\footnotetext{\emph{AMS 2020 subject classification}: 65C99, 60H99, 60G99}
}
\begin{abstract}
It is a well-established fact in the scientific literature that Picard iterations of backward stochastic
differential equations with globally Lipschitz continuous nonlinearity converge at least exponentially fast to the solution.
In this paper we prove that this convergence is in fact
at least square-root factorially fast.
We show for one example that no higher convergence speed is possible in general.
Moreover, if the nonlinearity is $z$-independent,
then the convergence is even factorially fast.
Thus we reveal a phase transition in the speed of convergence
of Picard iterations of backward stochastic differential equations.
\end{abstract}
\tableofcontents

\section{Introduction }

Since their introduction by Pardoux \& Peng in \cite{PardouxPeng1990} backward stochastic differential equations (BSDEs) 
have been extensively studied in the scientific literature and have found
numerous applications. For example, BSDEs provide a solution approach for stochastic optimal control problems, BSDEs appear in the pricing and hedging of options in mathematical finance, and BSDEs provide stochastic representations of semilinear parabolic partial differential equations (PDEs).

A standard approach for proving existence results for BSDEs is to construct a contraction mapping whose fixed point is the solution $(Y,Z)$ of the BSDE.
The associated fixed point iterations, the so-called Picard iterations,
are a key component of
several numerical approximation methods for BSDEs. We refer, e.g., to \cite{BenderDenk2007,bender2008time}
for numerical approximation methods for BSDEs based on Picard iterations and 
the least squares Monte Carlo method, 
we refer, e.g., to 
\cite{GobetLabart2010, LabartLelong2013}
for numerical approximation methods for BSDEs based on Picard iterations and adaptive control variates,
we refer, e.g., to 
\cite{BriandLabart2014, 
      GeissLabart2016}
for numerical approximation methods for BSDEs based on Picard iterations and Wiener chaos expansions, and we refer, e.g., to 
\cite{EHutzenthalerJentzenKruse2016, HJKNW2018, EHutzenthalerJentzenKruse2017,HutzenthalerKruse2017,hjk2019overcoming,beck2020overcoming,hutzenthaler2020multilevel} 
for numerical approximation methods for BSDEs based on Picard iterations
and a multilevel technique. Precise estimates on the speed of convergence of the Picard iterations $(Y_n,Z_n)_{n\in \N_0}$ to the solution $(Y,Z)$ of the BSDE are essential for the error analyses of these numerical approximation methods for BSDEs where 
$\N=\{1,2,\ldots\}$ and $\N_0=\N\cup\{0\}$.

Picard iterations, e.g., of ordinary differential equations converge not only exponentially fast but
even factorially fast under suitable assumptions.
Picard iterations of BSDEs are known to converge at least square-root factorially fast
if the nonlinearity is z-independent; see the proof of \cite[Theorem 3.1]{PardouxPeng1990}.
In the general case of z-dependent nonlinearities we have only found results
proving that Picard iterations converge at least 
exponentially fast 
(see, e.g., \cite[Theorem 2.1]{ElKarouiPengQuenez1997}, \cite[Theorem 4.3.1]{Zha17}, and \cite[Theorem 6.2.1]{pham2009continuous}).

In this article we prove for BSDEs with z-independent and 
globally Lipschitz continuous 
nonlinearities that the Picard iterations
converge in fact factorially fast.
Moreover, we show for BSDEs with z-dependent and globally Lipschitz continuous
nonlinearities
that the Picard iterations converge at least square-root factorially fast.
Somewhat surprisingly this speed of convergence cannot be improved in general.
More precisely, we establish for a linear example BSDE a corresponding lower
bound.
We thereby reveal a phase transition in the speed of convergence of Picard iterations between the z-independent and the z-dependent case. 
Theorem~\ref{t01} below illustrates the main results of this article.

\begin{theorem}\label{t01}
Let
$T\in (0,\infty)$,
 $d,m\in\N$, 
$\lipconstY,\lipconstZ \in [0,\infty) $,
$b\in \R^m$, 
let
$\lVert  \cdot\rVert \colon \cup_{n\in\N}\R^n\to[0,\infty)$
satisfy for all $n\in\N$ that $\lVert  \cdot\rVert|_{\R^n}$
is the standard norm on $\R^n$,
let
$\lVert  \cdot\rVert_{\sfF} \colon \R^{d\times m}\to[0,\infty)$ denote the Frobenius norm on $\R^{d\times m}$,
let $(\Omega,\mathcal{F},\P, (\F_t)_{t\in[0,T]}) $ be a filtered probability space which satisfies the usual conditions\footnote{Let $T \in (0,\infty)$ and let ${\bf \Omega} = (\Omega,\mathcal{F},\P, (\F_t)_{t\in[0,T]})$ be 
a filtered probability space. 
Then we say that ${\bf \Omega}$
satisfies the usual conditions if and only if 
it holds for all $t \in [0,T)$ that $\{ A\in \mathcal F: \P(A)=0 \} \subseteq \F_t 
= \cap_{ s \in (t,T] } \F_s$.}, 
let $f\colon[0,T]\times \Omega \times \R^d\times\R^{d\times m}\to\R^d$ be measurable,
assume for all  $t\in [0,T]$, $y,\tilde{y}\in\R^d$, $z,\tilde{z}\in\R^{d\times m}$ it holds a.s.\
  that
\begin{equation}  
\begin{split}
    \lVert f(t,y,z)-f(t,\tilde{y},\tilde{z})\rVert
    \leq \lipconstY \lVert y-\tilde{y}\rVert
    + \lipconstZ \lVert z-\tilde{z}\rVert_{\sfF},
\end{split} 
\end{equation}
let $W\colon [0,T]\times \Omega \to\R^m$ be a standard 
$(\F_t)_{t\in[0,T]}$-Brownian motion with continuous sample paths,
let $\xi \colon \Omega\to\R^d$ be $\F_T$-measurable,
let
$Y^k\colon[0,T]\times\Omega\to\R^d$, 
$k\in \N_0\cup\{\infty\}$,
be adapted
 with continuous sample paths,
let 
$Z^k\colon[0,T]\times\Omega\to\R^{d\times m}$,
$k\in \N_0\cup\{\infty\}$,
be progressively measurable, 
assume that
for all $s\in [0,T]$, $k\in\N_0\cup\{ \infty \}$
it holds  a.s.\  that
 $\int_0^T\E[\Vert\xi\rVert^2+\lVert f(t,0,0)\rVert^2+\lVert Y_t^\infty \rVert^2+\lVert Z_t^k\rVert^2_{\sfF}]\,dt<\infty$,
$Y^0_s=0$, $Z^0_s=0$,
and
\begin{align} 
Y^{k+1}_{s}=\xi+\int_{s}^{T}f(t,Y_t^k, Z_t^k)\,dt- \int_{s}^{T}
Z^{k+1}_t\,dW_t,
\end{align}
and let $e_k\in [0,\infty] $, $ k\in\N$, satisfy for all $ k\in\N$ that
\begin{align}
e_k=\left(\E\!\left[\sup_{t\in[0,T]}\left(\left\lVert Y^{k}_t-Y_t^\infty\right\rVert^2\right)+\int_0^{T}
\left\lVert Z^{k}_t-Z_t^\infty\right\rVert^2_{\sfF}\,dt\right]\right)^{\nicefrac{1}{2}}.
\end{align}
Then \begin{enumerate}[(i)]\itemsep0pt
\item \label{c01}there exists $c\in [0,\infty)$ such that for all $ k\in\N$ it holds
that
$
e_k
\leq \frac{c^k}{\sqrt{k!}},
$
\item\label{c05} if, in addition to the above assumptions, it holds that $\lipconstZ=0$,
then there exists $c\in [0,\infty)$ such that for all $ k\in\N$
it holds that
$e_k
\leq \frac{c^k}{k!},
$ 
and 
\item\label{c06}
 if, in addition to the above assumptions, 
$d=T=1$,
$\xi=2^{m/2}e^{-\frac{\lVert W_1\rVert^2}{2}}$,
and for all $t\in [0,T]$, $y\in \R$, $z\in \R^{1\times m}$ 
 it holds a.s.\ that $f(t,y,z)=z\cdot b$,
%
%
%
then there exists $c\in [0,\infty)$ such that for all $ k\in\N\cap[ \lVert b \rVert ^2 -1,\infty)$ it holds
that
$\frac{1}{2}
\left(\frac{ \lVert b \rVert ^2}{4}\right)^{\lfloor\frac{k+1}{2}\rfloor}
\frac{1}{\sqrt{k!}}\leq 
e_k
\leq  \frac{c^k}{\sqrt{k!}}.
$
\end{enumerate} 
\end{theorem}

Item~\eqref{c01} of \cref{t01} is a direct consequence of \cref{b20} and \cref{r01}. Item~\eqref{c05} of \cref{t01} follows from \cref{b20} and \cref{r02}.
Item~\eqref{c01} of \cref{t01} and \cref{a10}
prove Item~\eqref{c06} of \cref{t01}.
 The proof of 
Item~\eqref{c01} of \cref{t01}
is based on \cref{b19} 
which shows for all $t\in[0,T]$,  $k \in\N$, $\lambda\in(0,\infty)$
   that
\begin{align}\begin{split}
&
\E\!\left[e^{\lambda t}
\lvert Y^k_t-Y^\infty_t\rvert^2+\int_{t}^{T}
e^{\lambda s }
\lvert Z^k_s-Z^\infty_s\rvert^2\,ds
\right]\\
&
\leq \frac{1}{\lambda}
\E\!\left[\int_{t}^{T}e^{\lambda s}\lvert
f_s(Y_s^{k-1},Z_s^{k-1})
-f_s(Y_s^\infty,Z_s^\infty)
\rvert^2\,ds\right],\end{split}\label{l06}
\end{align}
  based on the Lipschitz continuity of $f$,  iterating \eqref{l06}  $k \in\N$ times
  and then setting $\lambda=k$ to get an upper bound of the form 
$c^k/ \sqrt{k^k}$ for 
$e_k$.


We finally discuss some possible consequences of Item~\eqref{c06} of \cref{t01} on the performance of numerical approximation methods for BSDEs based on Picard iterations in high-dimensional situations. 
To this end we consider a sequence of BSDEs indexed by the dimension $m \in \N$ of the driving Brownian motion $W$ whose associated Lipschitz constants $L_{\mathfrak{z},m} \in [0,\infty)$, $m\in \N$, grow for some $\alpha\in(0,\infty)$ like $m^\alpha$ as $m\to \infty$. Item~\eqref{c06} of \cref{t01} shows that it is possible in such a situation that the approximation errors $e_{k,m}$, $k,m\in \N$, grow faster in the dimension $m\in \N$ than any polynomial in the sense that 
for all $p\in [0,\infty)$ there exists $N\in \N$ such that for all $k\in \N\cap [N,\infty)$ it holds that $\liminf_{m\to \infty}\frac{e_{k,m}}{m^p}=\infty$.

The remainder of this article is organized as follows. In \cref{a09-1} we provide lower bounds for the convergence speed of Picard iterations. In \cref{a09+1} we establish in \cref{a10} lower bounds for the convergence speed of Picard iterations for a linear example BSDE. In our proof of  \cref{a10} we employ lower bounds for the convergence speed of Picard iterations for a linear example PDE which we prove in \cref{a09} in \cref{a09+-0}. In \cref{b19} in \cref{sec:apriori} we establish explicit a priori estimates for certain backward It\^o processes in appropriate $L^2$-norms. 
In \cref{sec:upper_bound} we provide upper bounds for the convergence speed of Picard iterations of BSDEs. \cref{b20} establishes an explicit bound for the $L^2$-distance between the Picard iterations and the solution of a BSDE with a globally Lipschitz continuous nonlinearity. In \cref{r01} we employ the estimate of \cref{b20} to obtain the square root-factorial speed of convergence of Picard iterations. In \cref{r02} we employ the estimate of \cref{b20} to obtain the factorial speed of convergence of Picard iterations in the z-independent case.

%
%
%

\section{Lower bounds for the convergence speed of Picard iterations}\label{a09-1}
In this section we provide lower bounds for the convergence speed of Picard iterations of BSDEs. 
In \cref{a09} in \cref{a09+-0} we establish lower bounds for the convergence speed of Picard iterations for a linear example PDE. We employ \cref{a09} in our proof of \cref{a10} in \cref{a09+1} to provide lower bounds for the convergence speed of Picard iterations for a linear example BSDE.
\cref{a10}
shows that square-root factorial convergence speed cannot be improved
  up to exponential factors in the case of $z$-dependent drivers.
Item \eqref{c05} of \cref{t01}
shows in the case of $z$-independent drivers that factorial speed of convergence is possible.
\cref{l01}
observes that factorial speed of convergence cannot be improved
up to exponential factors in the case of $y$-dependent drivers.

\subsection{Lower bounds for the convergence speed of Picard iterations for an example PDE}
\label{a09+-0}
\begin{lemma}\label{a09}
Let 
$d\in\N$, $b=(b_1,b_2,\ldots,b_d)\in \R^d$, 
let $\langle \cdot,\cdot\rangle \colon\R^d\times\R^d\to\R$ denote the standard scalar product on $\R^d$, let 
 $\lVert  \cdot\rVert \colon \R^d\to[0,\infty)$   denote the standard norm 
 on $\R^d$,
let $(\Omega,\mathcal{F},\P)$ be a probability space, 
let $W=(W^1,W^2,\ldots,W^d)\colon [0,1]\times \Omega \to \R^d$ be a standard Brownian motion,
 let 
$v^n\colon [0,1]\times \R^d\to \R$, $n\in\N_0\cup\{\infty\}$, satisfy for all
$t\in[0,1 ]$,
$x\in\R^d$, $n\in\N$ that
$v^0(t,x)=0$, 
\begin{align}\begin{split}
&v^n(t,x)=
\E\!\left[2^{d/2}
\exp\! \left(-\tfrac{\lVert x+W_1-W_t\rVert^2}{2}\right)
\right]
\\
&+
\sum_{k=1}^{n-1}
\sum_{\mu_1,\mu_2,\ldots,\mu_k=1 }^{d}
\Biggl[
\tfrac{(1-t)^{k}}{k!}
b_{\mu_1}
b_{\mu_2}
\cdots
b_{\mu_k}
\E\!\left[2^{d/2}
\tfrac{\partial^k}{\partial x_{\mu_1}\partial x_{\mu_2}\ldots\partial x_{\mu_k}}
\exp\! \left(-\tfrac{\lVert x+W_1-W_t\rVert^2}{2}\right)
\right]\Biggr],\label{a03}
\end{split}\end{align}
and 
\begin{equation}\label{a16}
v^\infty(t,x)=\E\!\left[2^{d/2}\exp \!\left(-\tfrac{\lVert x+b(1-t)+W_1-W_t\rVert^2}{2}\right)\right].
\end{equation}
Then
\begin{enumerate}[(i)]\itemsep0pt
\item\label{a19} it holds for all $t\in[0,1]$, $x\in\R^d$ that
$v^{\infty}\in C^{\infty}([0,1]\times\R^d,\R)$, $ v^\infty(1,x)= 2^{d/2}e^{-\frac{\lVert x\rVert^2}{2}}$, and
\begin{equation}
\frac{\partial v^{\infty}}{\partial t}(t,x)+
\frac{1}{2}(\Laplace_x v^\infty )(t,x)
+
\left\langle b,(\nabla_x v^\infty)(t,x)
\right\rangle=0,
\end{equation}\item\label{a13} it holds for all $n\in \N_0$, $t\in[0,1)$, $x\in\R^d$ that
$v^n \in C^{\infty}([0,1]\times\R^d,\R)$, $ v^n(1,x)= 2^{d/2}e^{-\frac{\lVert x\rVert^2}{2}}$, 
\begin{align}\begin{split}
v^{n+1}(t,x)&=\E\!\left[2^{d/2}\exp\!\left(-\frac{\lVert x+W_1-W_t\rVert^2}{2}\right)\right]\\
&\qquad+
\int_t^1\E\!\left[
\Bigl\langle b,
( \nabla_xv^n)(s,x+W_s-W_t)\Bigr\rangle
\right]\!ds,
\end{split}\end{align}
and
\begin{align}\begin{split}
 \nabla_x v^{n+1}(t,x)&=\E\!\left[2^{d/2}\exp\!\left(-\frac{\lVert x+W_1-W_t\rVert^2}{2}\right)\frac{W_1-W_t}{1-t}\right]\\
&\qquad+
\int_t^1\E\!\left[
\Bigl\langle b,
( \nabla_xv^n)(s,x+W_s-W_t)\Bigr\rangle
\frac{W_s-W_t}{s-t}\right]\!ds,
\end{split}\end{align}

\item \label{a14}
it holds for all $n\in\N$ that
\begin{align}
v^n(0,0)=1+\sum_{i=1}^{\lfloor\frac{n-1}{2}\rfloor}
\frac{(-1)^i \lVert b\rVert^{2i}}{4^ii!},
\end{align}
\item \label{a18}
it holds that
\begin{align}v^\infty(0,0)=\exp\!\left(-\frac{\lVert b\rVert^2}{4}\right),
\end{align}
and 
\item 
\label{a21}it holds for all $\epsilon\in (0,1)$,
$n\in\N\cap\bigl[\frac{1}{2\epsilon}\lVert b\rVert^2-1,\infty\bigr)$ that
\begin{align}
\left(\frac{\lVert b \rVert^2}{4}\right)^{\lfloor\frac{n+1}{2}\rfloor}\frac{1-\epsilon}{\lfloor\frac{n+1}{2}\rfloor!}
\leq 
\lvert
 v^\infty(0,0)- v^n(0,0)\rvert\leq 
\left(\frac{\lVert b \rVert^2}{4}\right)^{\lfloor\frac{n+1}{2}\rfloor}\frac{1}{\lfloor\frac{n+1}{2}\rfloor!}\frac{1}{1-\epsilon}.
\end{align}%
\end{enumerate}
\end{lemma}

\begin{proof}[Proof of \cref{a09}]
First note that the Feyman-Kac formula (cf., e.g., \cite[Theorem 8.2.1]{Oek03})
proves \eqref{a19}.

Next observe that 
\eqref{a03}
 proves for all $\ell\in\{1,2,\ldots,d\}$, $s\in[0,1)$, $x\in\R^d$, $n\in \N$ that $v^n \in C^{\infty}([0,1]\times\R^d,\R)$, $ v^n(1,x)= 2^{d/2}e^{-\frac{\lVert x\rVert^2}{2}}$,  and
\begin{align}\begin{split}
&
\tfrac{\partial v^n}{\partial x_{\ell}}(s,x)
-
\E\!\left[2^{d/2}\tfrac{\partial}{\partial x_{\ell}}
\exp\! \left(-\tfrac{\lVert x+W_1-W_s\rVert^2}{2}\right)
\right]=
\tfrac{\partial v^n}{\partial x_{\ell}}(s,x)
-
\tfrac{\partial}{\partial x_{\ell}}
\E\!\left[2^{d/2}
\exp\! \left(-\tfrac{\lVert x+W_1-W_s\rVert^2}{2}\right)
\right]
\\&=\sum_{k=1}^{n-1}
\sum_{\mu_1,\mu_2,\ldots,\mu_k=1 }^{d}
\Biggl[
\tfrac{(1-s)^{k}}{k!}
b_{\mu_1}
b_{\mu_2}
\cdots
b_{\mu_k}
\tfrac{\partial}{\partial x_{\ell}}
\E\!\left[2^{d/2}
\tfrac{\partial^k}{\partial x_{\mu_1}\partial x_{\mu_2}\ldots\partial x_{\mu_k}}
\exp\! \left(-\tfrac{\lVert x+W_1-W_s\rVert^2}{2}\right)
\right]\Biggr] \\
&=\sum_{k=1}^{n-1}
\sum_{\mu_1,\mu_2,\ldots,\mu_k=1 }^{d}
\Biggl[
\tfrac{(1-s)^{k}}{k!}
b_{\mu_1}
b_{\mu_2}
\cdots
b_{\mu_k}
\E\!\left[2^{d/2}
\tfrac{\partial^{k+1}}{\partial x_\ell\partial x_{\mu_1}\partial x_{\mu_2}\ldots\partial x_{\mu_k}}
\exp\! \left(-\tfrac{\lVert x+W_1-W_s\rVert^2}{2}\right)
\right]\Biggr].\label{a04}\end{split}
\end{align}
This, the disintegration theorem, and independence of Brownian increments
show for all $t\in[0,1)$, $s\in (t,1)$, $x\in\R^d$, $\ell\in\{1,2,\ldots,d\}$, $n\in\N$
that
\begin{align}
&\E\!\left[\tfrac{\partial v^n}{\partial x_{\ell}}(s,x+W_s-W_t)\right]
-
\E\!\left[2^{d/2}
\tfrac{\partial}{\partial x_\ell}
\exp\! \left(-\tfrac{\lVert x+W_1-W_t\rVert^2}{2}\right)
\right]\nonumber \\
&=\E\!\left[\tfrac{\partial v^n}{\partial x_{\ell}}(s,x+W_s-W_t)\right]
-
\E\!\left[\E\!\left[2^{d/2}
\tfrac{\partial}{\partial x_\ell}
\exp\! \left(-\tfrac{\lVert z+W_1-W_s\rVert^2}{2}\right)\right]\bigr|_{z=x+W_s-W_t}
\right]\nonumber \\
&
=\sum_{k=1}^{n-1}
\sum_{\mu_1,\mu_2,\ldots,\mu_k=1 }^{d}
\Biggl[
\tfrac{(1-s)^{k}}{k!}
b_{\mu_1}
b_{\mu_2}
\cdots
b_{\mu_k}\nonumber \\
&\qquad\qquad\qquad\qquad\qquad\cdot 
\E \biggl[\E\!\left[2^{d/2}
\tfrac{\partial^{k+1}}{\partial z_\ell\partial z_{\mu_1}\partial z_{\mu_2}\ldots\partial z_{\mu_k}}
\exp\! \left(-\tfrac{\lVert z+W_1-W_s\rVert^2}{2}\right)
\right]\bigr|_{z= x+W_s-W_t}\biggr]\Biggr]\nonumber \\
&=
\sum_{k=1}^{n-1}
\sum_{\mu_1,\mu_2,\ldots,\mu_k=1 }^{d}
\Biggl[
\tfrac{(1-s)^{k}}{k!}
b_{\mu_1}
b_{\mu_2}
\cdots
b_{\mu_k}
\E\!\left[2^{d/2}
\tfrac{\partial^{k+1}}{\partial x_{\mu_1}\partial x_{\mu_2}\ldots\partial x_{\mu_k}\partial x_\ell}
\exp\! \left(-\tfrac{\lVert x+W_1-W_t\rVert^2}{2}\right)
\right]\Biggr].\label{a11}
\end{align}
This, 
the fact that
$\forall\,k\in\N_0,t\in [0,1]\colon \int_{t}^{1}\frac{(1-s)^k}{k!}\,ds= \frac{(1-t)^{k+1}}{(k+1)!}$, and \eqref{a03} show for all
$t\in[0,1)$, $x\in\R^d$, $n\in\N$ that
\begin{align}
&\int_t^1\E\!\left[
\Bigl\langle b,
 (\nabla_x v^n)(s,x+W_s-W_t)\Bigr\rangle\right]\!ds=
\sum_{\mu_{k+1}=1}^{d}
\int_t^1\E\!\left[b_{\mu_{k+1}}\tfrac{\partial v^n}{\partial x_{\mu_{k+1}}}(s,x+W_s-W_t)\right]\!ds\nonumber \\
&=\sum_{k=0}^{n-1}
\sum_{\mu_1,\mu_2,\ldots,\mu_{k+1}=1 }^{d}
\int_{t}^{1}\tfrac{(1-s)^{k}}{k!}
b_{\mu_1}
b_{\mu_2}
\cdots
b_{\mu_{k+1}}
\E\!\left[2^{d/2}
\tfrac{\partial^{k+1}}{\partial x_{\mu_1}\partial x_{\mu_2}\ldots\partial x_{\mu_{k+1}}}
\exp\! \left(-\tfrac{\lVert x+W_1-W_t\rVert^2}{2}\right)
\right]\!ds\nonumber \\
&=\sum_{k=0}^{n-1}
\sum_{\mu_1,\mu_2,\ldots,\mu_{k+1}=1 }^{d}
\Biggl[\tfrac{(1-t)^{k+1}}{(k+1)!}
b_{\mu_1}
b_{\mu_2}
\cdots
b_{\mu_{k+1}}
\E\!\left[2^{d/2}
\tfrac{\partial^{k+1}}{\partial x_{\mu_1}\partial x_{\mu_2}\ldots\partial x_{\mu_{k+1}}}
\exp\! \left(-\tfrac{\lVert x+W_1-W_t\rVert^2}{2}\right)
\right]\Biggr]\nonumber \\
&=\sum_{k=1}^{n}
\sum_{\mu_1,\mu_2,\ldots,\mu_{k}=1 }^{d}
\Biggl[\tfrac{(1-t)^{k}}{k!}
b_{\mu_1}
b_{\mu_2}
\cdots
b_{\mu_{k}}
\E\!\left[2^{d/2}
\tfrac{\partial^{k}}{\partial x_{\mu_1}\partial x_{\mu_2}\ldots\partial x_{\mu_{k}}}
\exp\! \left(-\tfrac{\lVert x+W_1-W_t\rVert^2}{2}\right)
\right]\Biggr]\nonumber \\
&=  v^{n+1}(t,x)-\E\Bigl[2^{d/2}\exp \!\left(-\tfrac{\lVert x+W_1-W_t\rVert^2}{2}\right)\Bigr].\label{a12}
\end{align}
This, the fact that
$\forall\,t\in[0,1],x\in\R^d\colon v^0(t,x)= 0 $,
 and \eqref{a03}
show for all
$t\in[0,1)$, $x\in\R^d$, $n\in\N_0$ that
\begin{align}
v^{n+1}(t,x)=\E\Bigl[2^{d/2}\exp\!\left(-\tfrac{\lVert x+W_1-W_t\rVert^2}{2}\right)\Bigr]+
\int_t^1\E\!\left[
\Bigl\langle b,
( \nabla_xv^n)(s,x+W_s-W_t)\Bigr\rangle\right]\!ds.\label{a05}
\end{align}
Next note that
 Stein's lemma proves  for all 
$\ell\in \{1,2,\ldots,d\}$,
$x\in\R^d$, $s\in (0,1]$, $t\in[0,s)$, $h\in C^1(\R^d,\R)$ with
$\sup_{y\in\R^d} \bigl(\lvert h(y)\rvert+\lvert \tfrac{\partial h}{\partial y_\ell}(y)\rvert\bigr)<\infty$ 
 that
\begin{align}
&\frac{\partial }{\partial x_\ell}
\E\!\left [h(x+W_s-W_t)
\right]=\E\!\left [
\frac{\partial h}{\partial x_\ell}
(x+W_s-W_t)
\right]=
\E\!\left [
h(x+W_s-W_t) \frac{W^\ell_s-W^\ell_t}{s-t}
\right].
\end{align}
This, \eqref{a05}, differentiation under integrals,
the fact that
$\forall\, \ell\in \{1,2,\ldots, d\},n\in \N_0,s\in [0,1)\colon\sup_{x\in\R^d} \bigl[
\exp(-\frac{\lVert x\rVert^2}{2})+
\lvert  
\frac{\partial}{\partial x_{\ell}}
\exp (-\frac{\lVert x\rVert^2}{2})\rvert
+
\lvert \langle
b,(\nabla_x v^n)(s,x)
\rangle\rvert
+\lvert\frac{\partial}{\partial x_\ell} \langle
b,(\nabla_x v^n)(s,x)
\rangle\rvert\bigr]<\infty
$,
and the fact that $\forall\,t\in[0,1],x\in\R^d\colon v^0(t,x)=0$
show for all 
$\ell\in \{1,2,\ldots,d\}$, $n\in \N_0$, $t\in[0,1)$, $x\in\R^d$ that
\begin{align}
&\frac{\partial v^{n+1}}{\partial x_\ell}(t,x)=\frac{\partial}{\partial x_\ell}
\E\!\left[2^{d/2}\exp\!\left(-\tfrac{\lVert x+W_1-W_t\rVert^2}{2}\right)\right]+
\int_t^1\frac{\partial}{\partial x_\ell} \E\!\left[
\Bigl\langle b,(\nabla_x v^n)
(s,x+W_s-W_t)\Bigr\rangle\right]\!ds\nonumber \\
&=
\E\!\left[2^{d/2}\exp \!\left(-\tfrac{\lVert x+W_1-W_t\rVert^2}{2}\right)\frac{W_1^\ell-W_t^\ell}{1-t}\right]\nonumber \\
&\qquad\qquad\qquad+
\int_t^1\E\!\left[
\Bigl\langle b,
( \nabla_x v^n) (s,x+W_s-W_t)\Bigr\rangle\frac{W^\ell_s-W^\ell_t}{s-t}\right]\!ds.
\end{align}
This and 
\eqref{a05} show \eqref{a13}. 

For the next step let $0^0=1$,
let  $H_k\colon \R \to \R$, $k\in \N_0$, satisfy for all $k\in\N_0$, $x\in\R$ that 
\begin{equation}\label{a02}
H_k(x)= \sum_{\ell =0}^{\lfloor \frac{k}{2}\rfloor}
\left[\frac{k!(-1)^\ell }{\ell !(k-2\ell)!}\frac{x^{k-2\ell}}{2^\ell }\right]
\end{equation}
and
for every $n\in \N_0\cup\{-1\}$ let $n!!\in\N$ satisfy that
$ n!!=\prod _{k=0}^{\left\lceil { {n}/{2}}\right\rceil -1}(n-2k)$.
A well-known fact on Hermite polynomials shows for all
$x\in\R$, $k\in\N_0$ that
$
\tfrac{d^k}{dx^k}(e^{-\frac{x^2}{2}})= 
e^{-\frac{x^2}{2}} H_k(x).
$ Furthermore, a well-known fact  on 
moments of
normally distributed random variables 
shows
for all 
$k\in \N_0$, $\ell\in [0,k]\cap\N_0$
that
$
 \tfrac{1}{\sqrt{\pi}}\int_{\R}
z^{2k+1-2\ell}e^{-z^2}\,dz= 0
$ and \begin{align}
 \frac{1}{\sqrt{\pi}}\int_{\R}
z^{2k-2\ell}e^{-z^2}\,dz= \left[ \frac{1}{\sqrt{2\pi\sigma^2}}\int_{\R}
z^{2k-2\ell}e^{-\frac{z^2}{2\sigma^2}}\,dz\right]\Bigr|_{\sigma^2=\frac{1}{2}}
=\frac{(2k-2\ell-1)!!}{2^{k-\ell}}.
\end{align}
This, \eqref{a02}, and the binomial theorem imply
for all  $k\in \N_0$ that
\begin{align}
&
\E\! \left[\sqrt{2}\exp \!\left(-\frac{\lvert W_1^1\rvert^2}{2}\right)
H_{2k}(W_1^1)
\right]= \frac{1}{\sqrt{2\pi}}\int_{\R}\sqrt{2}e^{-\frac{z^2}{2}}
\left[\sum_{\ell=0}^{k}\left(\frac{(2k)!(-1)^\ell}{\ell!(2k-2\ell)!}\frac{z^{2k-2\ell}}{2^\ell}\right)\right]e^{-\frac{z^2}{2}}\,dz\nonumber \\
&=\sum_{\ell=0}^{k}\left[\frac{(2k)!(-1)^\ell}{\ell!(2k-2\ell)!2^\ell}
\left(
 \frac{1}{\sqrt{\pi}}\int_{\R}
z^{2k-2\ell}e^{-z^2}\,dz\right)\right]
=
\sum_{\ell=0}^{k}\left[
\frac{(2k)!(-1)^\ell}{\ell!(2k-2\ell)!2^\ell}
\frac{(2k-2\ell-1)!!}{2^{k-\ell}}\right]
\nonumber \\
&=
\sum_{\ell=0}^{k}\frac{(2k)!(-1)^\ell(2k-2\ell-1)!!}{\ell!(2k-2\ell)!2^k}\nonumber 
=\sum_{\ell=0}^{k}\frac{(2k)!(-1)^\ell}{\ell!(2k-2\ell)!!2^k}\\
&
=
\sum_{\ell=0}^{k}\frac{(2k)!(-1)^\ell}{\ell!(k-\ell)!2^{k-\ell}2^k}=\frac{(2k)!}{4^kk!}\sum_{\ell=0}^{k}\frac{k!(-2)^\ell}{(k-\ell)!\ell!}=\frac{(2k)!}{4^kk!}(-1)^k 
\label{a07}\end{align}
and
\begin{align} \label{a08}\begin{split}
&
\E\! \left[\sqrt{2}\exp \!\left(-\frac{\lvert W_1^1\rvert^2}{2}\right)
H_{2k+1}(W_1^1)
\right]\\&= \frac{1}{\sqrt{2\pi}}\int_{\R}\sqrt{2}e^{-\frac{z^2}{2}}
\left[\sum_{\ell=0}^{\lfloor (2k+1)/{2}\rfloor}\left(\frac{(2k+1)!(-1)^\ell}{\ell!(2k+1-2\ell)!}\frac{z^{2k+1-2\ell}}{2^\ell}\right)\right]e^{-\frac{z^2}{2}}\,dz=0.
\end{split}\end{align}
Thus, it holds for all $k\in \N_0$ that
\begin{equation}\label{a08b}
\E\! \left[\sqrt{2}\exp\!\left(-\frac{\lvert W_1^1\rvert^2}{2}\right)
H_{k}(W_1^1)
\right]=\1_{2\N_0}(k)\frac{k!(-1)^{\lfloor k/2\rfloor}}{4^{k/2} \lfloor k/2\rfloor!  }.
\end{equation}
Furthermore, the combinatorial interpretation of multinomial coefficients
yields for all $k\in\N$, $\alpha\in (\N_0)^k$ with $\lvert \alpha\rvert=k$
that
\begin{align}
\begin{split}
&\#\left[\bigcap_{i=1}^d\Bigl\{ (\mu_1,\mu_2,\ldots,\mu_k)\in \{1,2,\ldots,d\}^k\colon  \# \{\ell\in \{1,2,\ldots,k\}\colon \mu_\ell=i\}=\alpha_i
\Bigr\}\right]\\
&=\frac{k!}{\alpha_1!\alpha_2!\cdots\alpha_d!}.
\end{split}\end{align}
This,  the fact that $W^1,W^2,\ldots,W^d$ are independent,
\eqref{a03},   and
the fact that
$\forall\, x\in\R, k\in\N_0\colon
\frac{d^k}{dx^k}(e^{-\frac{x^2}{2}})= 
e^{-\frac{x^2}{2}} H_k(x)
$
 show for all $n\in \N$ that
\begin{align}
&v^n(0,0)-\left[
\prod_{\ell=1}^{d}
\E\!\left[\sqrt{2}
\exp \!\left(-\frac{(W_1^\ell)^2}{2}\right)
\right]\right]=
v^n(0,0)-
\E\!\left[2^{d/2}
\exp\! \left(-\frac{\lVert W_1\rVert^2}{2}\right)
\right]
\nonumber \\
&=
\sum_{k=1}^{n-1}
\sum_{\mu_1,\mu_2,\ldots,\mu_k=1 }^{d}
\Biggl[
\frac{1}{k!}
b_{\mu_1}
b_{\mu_2}
\cdots
b_{\mu_k}
\E\!\left[2^{d/2}
\frac{\partial^k}{\partial x_{\mu_1}\partial x_{\mu_2}\ldots\partial x_{\mu_k}}
\exp\! \left(-\frac{\lVert x+W_1\rVert^2}{2}\right)
\right]\Bigr|_{x=0}\Biggr] \nonumber  \\
&
=
\sum_{k=1}^{n-1}
\sum_{\alpha\in (\N_0)^d\colon \lvert \alpha \rvert=k}\left[
\prod_{\ell=1}^{d}
\E\!\left[\frac{b_{\ell}^{\alpha_{\ell}}\sqrt{2}}{\alpha_\ell!}
\frac{\partial^{\alpha_{\ell}}}{\partial x_\ell^{\alpha_{\ell}}}
\exp\! \left(-\frac{(x_\ell+W_1^\ell)^2}{2}\right)
\right]\Bigr|_{x_\ell=0}\right] \nonumber \\
&=\sum_{k=1}^{n-1}
\sum_{\alpha\in (\N_0)^d\colon \lvert \alpha \rvert=k}\left[
\prod_{\ell=1}^{d}
\E\!\left[\frac{b_{\ell}^{\alpha_{\ell}}\sqrt{2}}{\alpha_\ell!}
\exp \!\left(-\frac{(W_1^\ell)^2}{2}\right)
H_{\alpha_\ell}(W_1^\ell)
\right]\right] .\label{m01}
\end{align}
This,
\eqref{a08b}, and the multinomial theorem 
 show for all $n\in \N$ that
\begin{align}
&v^n(0,0)=\sum_{k=0}^{n-1}
\sum_{\alpha\in (\N_0)^d\colon \lvert \alpha \rvert=k}
\prod_{\ell=1}^{d}\left[
\frac{b_{\ell}^{\alpha_{\ell}}}{\alpha_\ell!}
\1_{2\N_0}(\alpha_\ell)
\frac{\alpha_\ell!(-1)^{\lfloor \alpha_\ell/2\rfloor} }{4^{\alpha_\ell/2} \lfloor \alpha_\ell/2\rfloor! }\right]\nonumber \\
&=
\sum_{i=0}^{\lfloor\frac{n-1}{2}\rfloor}
\sum_{\beta\in (\N_0)^d\colon \lvert \beta\rvert=i}
\prod_{\ell=1}^{d}
\frac{b_{\ell}^{2\beta_{\ell}}(-1)^{ \beta_\ell} }{4^{\beta_\ell}  \beta_\ell! } =\sum_{i=0}^{\lfloor\frac{n-1}{2}\rfloor}\left[
\frac{(-1)^i }{4^i}
\sum_{\beta\in (\N_0)^d\colon \lvert \beta\rvert=i}
\prod_{\ell=1}^{d}
\frac{b_{\ell}^{2\beta_{\ell}}}{ \beta_\ell!  }
\right]\nonumber \\
&
= \sum_{i=0}^{\lfloor\frac{n-1}{2}\rfloor}\left[
\frac{(-1)^i }{4^ii!}
\sum_{\beta\in (\N_0)^d\colon \lvert \beta\rvert=i}\left[
\frac{i!b_1^{2\beta_1}b_2^{2\beta_2}\ldots b_d^{2\beta_d}}{\beta_1!\beta_2!\cdots\beta_d!}\right]\right] =
 \sum_{i=0}^{\lfloor\frac{n-1}{2}\rfloor}
\frac{(-1)^i \lVert b\rVert^{2i}}{4^ii!}
. \label{m02}
\end{align}
This establishes \eqref{a14}.

Next observe that 
for all $a\in\R$, $\ell\in \{1,2,\ldots,d\}$ it holds that
\begin{equation}\begin{split}
&\E\!\left[ \exp \!\left(-\frac{( a+W^\ell_1)^2}{2}\right)\right]
=
\frac{1}{\sqrt{2\pi}}\int_{\R} e^{-\frac{(a+z)^2}{2}}e^{-\frac{z^2}{2}}\,dz
=\frac{1}{\sqrt{2\pi}}\int_{\R} e^{-\frac{a^2}{2}}
e^{-az}
e^{-z^2}\,dz\\
&
=
\frac{1}{\sqrt{2\pi}}\int_{\R} e^{-\frac{a^2}{2}+\frac{1}{4}a^2}
e^{ -(z+\frac{1}{2}a)^2}
\,dz
=\frac{1}{\sqrt{2\pi}}\int_{\R}e^{-\frac{a^2}{4}}e^{-\frac{y^2}{2}}\frac{1}{\sqrt{2}}\,dy= \frac{e^{-\frac{a^2}{4}}}{\sqrt{2}}.\end{split}
\end{equation}
This, \eqref{a16}, and the fact that $W^1,W^2,\ldots,W^d$ are independent 
 prove that
\begin{equation}
v^\infty(0,0)=\E\!\left[2^{d/2}\exp\! \left(-\frac{\lVert b+W_1\rVert^2}{2}\right)\right]= 2^{d/2}\prod_{\ell=1}^{d}\E\!\left[ \exp\!\left(-\frac{( b_\ell+W^\ell_1)^2}{2}\right)\right]=e^{-\frac{\lVert b \rVert^2}{4}} 
.
\end{equation}
This shows \eqref{a18}. 

Next note that \eqref{a14}, \eqref{a18}, and the fact that
$\forall\,x\in\R\colon e^{-\frac{x^2}{4}}=1+\sum_{i=1}^\infty \frac{(-1)^i x^{2i}}{4^ii!}$
 show  for all $n\in\N$ that
\begin{align}
 v^\infty(0,0)- v^n(0,0)=\sum^ {\infty}_{i=\lfloor \frac{n-1}{2}\rfloor+1}\frac{(-1)^{i}\lVert b\rVert^{2i}}{i! 4^{i}}=
\sum^ {\infty}_{i=\lfloor \frac{n+1}{2}\rfloor}\frac{(-1)^{i}\lVert b\rVert^{2i}}{ i! 4^{i}}
.\label{m03}
\end{align}
Therefore,
for all $\epsilon\in (0,1)$,
$n\in\N\cap\bigl[\frac{1}{2\epsilon}\lVert b\rVert^2-1,\infty\bigr)$, $j\in \N$ with 
$j
=\lfloor \tfrac{n+1}{2}\rfloor
$
it holds that 
$\frac{1}{2\epsilon}\lVert b\rVert^2\leq n+1=2\frac{n+1}{2}
\leq 2(\lfloor \frac{n+1}{2}\rfloor +1)
= 2(j+1)$, 
$\frac{\lVert b\rVert^2}{4(j+1)}\leq \epsilon $,
 \begin{align}
&
( v^\infty(0,0) -v^n(0,0) ) (-1)^j= 
\left(\frac{\lVert b\rVert^{2j}}{j!4^j}- \frac{\lVert b \rVert^{2(j+1)}}{(j+1)!4^{j+1}}\right)
+ \left(\frac{\lVert b\rVert^{2(j+2)}}{(j+2)!4^{j+2}}- \frac{\lVert b \rVert^{2(j+3)}}{(j+3)!4^{j+3}}\right)+\ldots\nonumber \\
&=  \frac{\lVert b\rVert^{2j}}{j!4^j}\left(1- \frac{\lVert b\rVert^2}{4(j+1)}\right)
+
\frac{\lVert b\rVert^{2(j+2)}}{(j+2)!4^{j+2}}
\left(1- \frac{\lVert b\rVert^2}{4(j+3)}\right)
+\ldots\nonumber \\
&
\geq \frac{\lVert b\rVert^{2j}}{j!4^j}\left(1- \frac{\lVert b\rVert^2}{4(j+1)}\right)
\geq 
\left(\frac{\lVert b \rVert^2}{4}\right)^{j}\frac{1}{j!}(1-\epsilon),
\end{align}
and
\begin{align}\begin{split}
&\lvert
 v^\infty (0,0)- v^n(0,0)\rvert\\
&\leq 
\sum^ {\infty}_{i=j}\frac{\lVert b\rVert^{2i}}{i! 4^{i}}\leq 
\frac{\lVert b\rVert^{2j}}{j! 4^{j}}\left(
1+\sum_{i=j+1}^{\infty}\left[\frac{\lVert b\rVert^{2}}{4(j+1)}\right]^{^{i-j}}
\right)\leq \frac{\lVert b\rVert^{2j}}{j! 4^{j}}\sum_{\ell=0}^{\infty} \epsilon^\ell
=\left(\frac{\lVert b \rVert^2}{4}\right)^{j}\frac{1}{j!}\frac{1}{1-\epsilon}.\end{split}
\end{align}
This  establishes \eqref{a21}.
The proof of \cref{a09} is thus completed.
\end{proof}
\subsection{Lower bounds for the convergence speed of Picard iterations for an example BSDE}\label{a09+1}
\begin{corollary}\label{a10}
Let 
$d\in\N$, $b\in \R^d$, 
let $\langle \cdot,\cdot\rangle \colon\R^d\times\R^d\to\R$ denote the standard scalar product on $\R^d$, let 
 $\lVert  \cdot\rVert \colon \R^d\to[0,\infty)$   denote the standard norm  on $\R^d$,
let $(\Omega,\mathcal{F},\P, (\F_t)_{t\in[0,1]})$ be a filtered probability space which satisfies the usual conditions,
let $W\colon [0,1]\times \Omega \to \R^d$ be a standard 
$(\F_t)_{t\in[0,1]}$-Brownian motion with continuous sample paths,
let  $Y^n\colon[0,1]\times\Omega\to\R$, $n\in \N_0\cup\{\infty\}$, be adapted
   with continuous sample paths,
let
$ Z^n\colon[0,1]\times\Omega\to\R^{d}$, $n\in\N_0\cup\{\infty\}$,
   be progressively
   measurable,
 and assume for all $s\in [0,1]$, $n\in\N\cup\{\infty\}$ that
 a.s.\ it holds that
$Y^0_s=0$,  $Z^0_s=0$, 
$\int_{0}^{T}\E \bigl[ \lVert Z^n_t\rVert^2\bigr]\, dt <\infty$ and
\begin{align}\label{t03}\begin{split}
Y_s^{n+1}= 2^{d/2}e^{-\frac{\lVert W_1\rVert^2}{2}}+ \int_{s}^1 \langle b, Z^n_t\rangle\,dt-\int_s^1 \langle Z^{n+1}_t,dW_t\rangle.
\end{split}\end{align}
Then for all $n\in\N\cap\bigl[\lVert b\rVert^2-1,\infty\bigr)$ 
it holds  a.s.\ that
 $\left\lvert Y_0^{\infty}-Y^n_0\right\rvert\geq\frac{1}{2}
\left(\frac{\lVert b \rVert^2}{4}\right)^{\lfloor\frac{n+1}{2}\rfloor}
\frac{1}{\sqrt{n!}}$.
\end{corollary}
\begin{proof}[Proof of \cref{a10}]Throughout this proof let
$v^n\colon [0,1]\times \R^d\to \R$, $n\in\N_0\cup\{\infty\}$, satisfy for all
$t\in[0,1 ]$,
$x\in\R^d$, $n\in\N$ that
$v^0(t,x)=0 $,
\begin{align}
v^\infty(s,x)=\E\!\left[2^{d/2}\exp \!\left(-\frac{\lVert x+b(1-s)+W_1-W_s\rVert^2}{2}\right)\right], \label{t05}
\end{align}
and
\begin{align}
\begin{split}
&v^n(s,x)=
\E\!\left[2^{d/2}
\exp\! \left(-\frac{\lVert x+W_1-W_s\rVert^2}{2}\right)
\right]
\\
&+
\sum_{k=1}^{n-1}
\sum_{\mu_1,\mu_2,\ldots,\mu_k=1 }^{d}
\Biggl[
\frac{(1-s)^{k}}{k!}
b_{\mu_1}
b_{\mu_2}
\cdots
b_{\mu_k}
\E\!\left[2^{d/2}
\frac{\partial^k}{\partial x_{\mu_1}\partial x_{\mu_2}\ldots\partial x_{\mu_k}}
\exp\! \left(-\frac{\lVert x+W_1-W_s\rVert^2}{2}\right)
\right]\Biggr].\label{a03b}
\end{split}
\end{align}
Then \cref{a09} proves
\begin{enumerate}[a)]
\item for all $t\in[0,1] $, $x\in\R^d$
that
$v^\infty\in C^{\infty}([0,1]\times\R^d,\R)$ and 
\begin{equation}
\frac{\partial v^{\infty}}{\partial t}(t,x)+
\frac{1}{2}(\Laplace_x v^\infty )(t,x)
+
\left\langle b,(\nabla_x v^\infty)(t,x)
\right\rangle=0,
\end{equation}
\item for all $s\in[0,1]$,
$x\in\R^d$, 
$n\in\N_0$ that $v^n\in C^{\infty}([0,1]\times\R^d,\R)$ and
\begin{equation}
v^{n+1}(s,x)=\E\!\left[2^{d/2}\exp\!\left(-\tfrac{\lVert x+W_1-W_s\rVert^2}{2}\right)\right]+
\int_s^1\E\!\left[
\Bigl\langle b,
( \nabla_xv^n)(t,x+W_t-W_s)\Bigr\rangle\right]\!dt,\label{t07a}
\end{equation}
and 
\item for all $n\in\N\cap\bigl[\lVert b\rVert^2-1,\infty\bigr)$
that
\begin{equation}
\lvert v^n(0,0)-v^\infty(0,0)\rvert
\geq \frac{1}{2}
\left(\frac{\lVert b \rVert^2}{4}\right)^{\lfloor\frac{n+1}{2}\rfloor}\frac{1}{\lfloor\frac{n+1}{2}\rfloor!}.\label{t07}
\end{equation}
\end{enumerate}
This and It\^o's formula
prove that for all $s\in [0,1]$   it holds a.s.\ that
\begin{align}\begin{split}
&
2^{d/2} e^{-\frac{\lVert W_1\rVert^2}{2}}-v^\infty (s,W_s)=v^{\infty}(1,W_1)-v^{\infty}(s,W_s)\\
&= \int_{s}^{1} \left(\frac{\partial v^{\infty}}{\partial t}+
\frac{1}{2}\Laplace_x v^\infty\right)\!(t,W_t)\,dt
+\int_{s}^{1} \left\langle (\nabla_x v^{\infty})(t,W_t), dW_t\right\rangle
\\
&=-\int_{s}^{1} \left\langle b,(\nabla_x v^{\infty}) (t,W_t)\right\rangle dt
+\int_{s}^{1} \left\langle (\nabla_x v^{\infty}) (t,W_t), dW_t\right\rangle.
\end{split}
\end{align}
This, \eqref{t03}, and 
a standard result on
uniqueness of backward stochastic differential equations
(cf., e.g., \cite[Theorem~4.3.1]{Zha17})
prove 
for all
$s\in [0,1]$
that
$\P\bigl(  Y^{\infty}_s= v^{\infty}(s,W_s)\bigr)=1$
and 
$\P\bigl(  Z^{\infty}_s= (\nabla_xv^{\infty})(s,W_s)\bigr)=1$.

Next, we prove by induction on  $n\in\N_0$ that for all $ n\in\N_0$, $s\in [0,1]$ it holds that
$\P\bigl(  Y^{n}_s= v^{n}(s,W_s)\bigr)=1$
and 
$\P\bigl(  Z^{n}_s= (\nabla_xv^{n})(s,W_s)\bigr)=1$.
 First, 
the fact that  $\forall \,t\in[0,1],x\in\R^d\colon v^0(t,x)=0 $ and
the fact that 
$\forall\, s \in [0,1]\colon\P \bigl( (Y^0_s,  Z^0_s)=(0,0)\bigr)=1 $
 establish the base case $n=0$. For the induction step
$\N_0\ni n\mapsto n+1\in\N$ let $n\in\N_0$ satisfy for all
$s\in [0,1]$
 that
$\P\bigl(  Y^{n}_s= v^{n}(s,W_s)\bigr)=1$
and 
$\P\bigl(  Z^{n}_s= (\nabla_xv^{n})(s,W_s)\bigr)=1$.
This,  \eqref{t07a}, the Markov property of $W$, 
the fact that for all
$s\in [0,1]$
it holds a.s.\ that $
\E\bigl[\int_s^1 \langle
{Z}^{n+1}_t,dW_t\rangle|\F_s\bigr]=0$,
\eqref{t03}, and adaptedness of
$Y^{n+1}$ 
imply that for all $s\in[0,1]$ it holds a.s.\  that
\begin{equation}\begin{split}
&v^{n+1}(s,W_s)=\E\Bigl[2^{d/2}e^{-\frac{\lVert W_1\rVert^2}{2}}\big|\F_s\Bigr]+
\int_s^1\E\!\left[
\Bigl\langle b,
( \nabla_xv^n)(t,W_t)\Bigr\rangle\big|\F_s\right]\!dt\\
&= 
\E\!\left[2^{d/2}e^{-\frac{\lVert W_1\rVert^2}{2}}\big|\F_s\right] + 
\int_{s}^{1}\E\bigl[ \langle b,{Z}^n_t\rangle \big|\F_s\bigr]dt\\
&
=\E\!\left[2^{d/2}e^{-\frac{\lVert W_1\rVert^2}{2}} + \int_{s}^{1} \langle b,Z^n_t\rangle\,dt+
\int_s^1 \langle {Z}^{n+1}_t,dW_t\rangle
\middle|\F_s\right]= \E \!\left[Y_{s}^{n+1}\middle|\F_s\right]=Y_{s}^{n+1}.\end{split}\label{t08}
\end{equation}
This, 
It\^o's formula, 
the fact that
$v^{n+1}\in C^{2}([0,1]\times\R^d,\R)$,
\eqref{a03b},
 and
\eqref{t03} show that for all
$s\in[0,1]$ it holds a.s. that
\begin{align}\label{a20}\begin{split}
&
0= v^{n+1}(s,W_s) - Y_s^{n+1}
= - \int_{s}^{1}
\left[\left(  \frac{\partial v^{n+1}}{\partial t}
+\frac{1}{2}\Laplace_x v^{n+1} \right)\!(t,W_t) -\left\langle b,Z^{n}_t\right\rangle\right]dt \\
&\qquad\qquad\qquad-
\int_{s}^{1}\left\langle(\nabla_x v^{n+1}) (t,W_t)-Z^{n+1}_t ,dW_t\right\rangle
.
\end{split}\end{align}
This and the uniqueness of the decomposition of continuous semimartingales show
for all $s\in [0,1]$ that
$\P\bigl( Y^{n+1}_s= v^{n+1}(s,W_s)\bigr)=1$
and 
$\P\bigl(  Z^{n+1}_s= (\nabla_xv^{n+1})(s,W_s)\bigr)=1$. This completes the induction step.
Induction thus shows for all $ n\in\N_0$, $s\in [0,1]$ that
$\P\bigl( Y^{n}_s= v^{n}(s,W_s)\bigr)=1$
and 
$\P\bigl(  Z^{n}_s= (\nabla_xv^{n})(s,W_s)\bigr)=1$.
This and the fact that
$\P\bigl(  Y^{\infty}_0= v^{\infty}(0,W_0)\bigr)=1$
imply that for all $n\in \N$ it holds a.s.\ that
$Y^n_0-Y_0^{\infty}= v^n(0,0)-v^{\infty}(0,0)$.
This, the fact that for all $k,n\in \N_0$ with $n=2k$ it holds that
\begin{equation}\begin{split}
&
\left\lfloor\frac{n+1}{2}\right\rfloor != \left\lfloor\frac{2k+1}{2}\right\rfloor !
= k!\\
&= 1\cdot 2\cdots k \leq 
\sqrt{1\cdot 2\cdots k\cdot (k+1)(k+2)\cdots(2k)}=
\sqrt{(2k)!}=\sqrt{n!},\end{split}
\end{equation}
the fact that
for all $k,n\in \N_0$ with $n=2k+1$ it holds that
\begin{equation}\begin{split}
&
\left\lfloor\frac{n+1}{2}\right\rfloor!
=\left\lfloor\frac{2k+1+1}{2}\right\rfloor!= (k+1)!\\
&
=  2\cdot 3 \cdots (k+1) \leq 
\sqrt{2\cdot 3 \cdots (k+1)(k+2)(k+3)\cdots(2k+1) }
= \sqrt{n!},
\end{split}\end{equation}
and \eqref{t07}
imply that for all $n\in\N\cap\bigl[\lVert b\rVert^2-1,\infty\bigr)$  
  it holds a.s.\ that
\begin{equation}
\left\lvert Y_0^n-Y_0^{\infty}\right\rvert
= \left\lvert v^n(0,0)-v^{\infty}(0,0)\right\rvert
\geq 
\frac{1}{2}
\left(\frac{\lVert b \rVert^2}{4}\right)^{\lfloor\frac{n+1}{2}\rfloor}
\frac{1}{\sqrt{n!}}
. 
\end{equation}
This completes the proof of \cref{a10}.
\end{proof}
 The following \cref{l01} 
gives an example where the BSDE solution is an expoential function (see Item~\eqref{l01a})
  and the Picard approximations are just partial sums of the exponential series (see Item~\eqref{l01b}).
  Thereby, \cref{l01} shows that factorial speed of convergence cannot be improved
  up to exponential factors in the case of $y$-dependent drivers.
\begin{lemma}\label{l01}
Let $T\in (0,\infty)$ and
let $Y^k\in C( [0,T],\R) $, $k\in\N_0\cup\{\infty\}$,
satisfy for all $k\in\N_0$,
$s\in [0,T]$
 that
$Y_s^0 =1$,
$Y_s^{k+1}=1+\int_s^T Y_r^k \,dr$, and  
$Y_s^\infty= 1+ \int_s^T Y_r^\infty dr$.
Then \begin{enumerate}[i)] \setlength{\itemsep}{0pt}
\item\label{l01a} it holds for all $s\in[0,T]$ that $Y_s^\infty = e^{T-s} $,
\item \label{l01b}it holds for all $s\in[0,T]$, $n\in\N_0$ that  $Y_s^n =1+ \sum_{k=1}^n \frac{(T-s)^k}{ k! }$, and 
\item\label{l01c} it holds that
 $\sup_{s \in [0,T]} \lvert Y_s^\infty - Y_s^n\rvert= \sum_{k=n+1}^{\infty} \frac{T^k}{k!} \geq \frac{T^{n+1}}{(n+1)!}$.
\end{enumerate}
\end{lemma}
\begin{proof}
[Proof of \cref{l01}] The fact that
$\forall\, s\in [0,T]\colon Y_s^\infty= 1+ \int_s^T Y_r^\infty \,dr$ and the substitution rule show for all
$s\in[0,T]$ that $Y^\infty_{T-s}= 1+ \int_{T-s}^{T}Y_r^\infty\,dr= 1+\int_{0}^{s}Y_{T-r}^\infty\,dr$. This, the fact that
 $\forall\,s\in[0,T]\colon e^{T-s}= 1+\int_{0}^{s}e^{T-r}\,dr$, and the Picard--Lindel\"of theorem show \eqref{l01a}.
Next, we prove \eqref{l01b} by induction on $n\in\N_0$. 
 The fact that $\forall\,s\in [0,T]\colon Y_s^0 =1$
shows the base case  $n=0$. For the induction step $\N_0\ni n\mapsto n+1\in\N $ let $n\in\N_0$ and assume for all $s\in[0,T]$ that  $Y_s^n =1+ \sum_{k=1}^n \frac{(T-s)^k}{ k! }$.
The assumptions of \cref{l01} then
 show for all $s\in[0,T]$ that
\begin{align}\begin{split}
&
Y_s^{n+1}=1+\int_s^T Y_r^n \,dr=
1+\int_{s}^{T}\left(1+\sum_{k=1}^n \frac{(T-r)^k}{ k! }\right)dr\\
&
= 1+\sum_{k=1}^{n+1} \frac{-(T-r)^{k+1}}{(k+1)!}\Bigr|_{r=s}^T=1+\sum_{k=1}^{n+1}
\frac{(T-s)^k}{k!}.
\end{split}\end{align}
This completes the induction step. Induction hence shows \eqref{l01b}. Combining \eqref{l01a}, \eqref{l01b}, and the fact that
$\forall\,x\in\R\colon e^x=1+\sum_{k=1}^{\infty}\frac{x^k}{k!}$ yields \eqref{l01c}. The proof of \cref{l01} is thus completed.
\end{proof}

\section{A priori estimates for backward It\^o processes}\label{sec:apriori}
In this section we establish a priori estimates for certain backward It\^o processes. Results of this form are well-known in the scientific literature on BSDEs
(see, e.g., 
\cite[Proof of Theorem 3.1]{PardouxPeng1990},
\cite[Proposition 2.1]{ElKarouiPengQuenez1997},
\cite[Theorem 4.2.1]{Zha17},
\cite[Proposition 5.2]{pardouxrascanu2016}).
\cref{b19} below establishes estimates for an It\^o process
and its diffusion process in terms of the drift process and in terms
of the terminal value of the It\^o process.
The contribution of \cref{b19} is to provide explicit universal constants.
Moreover, the It\^o process in \cref{b19} and its drift process are not
assumed to be square-integrable and, in particular, the right-hand sides
of \eqref{eq:b01}, \eqref{eq:b02}, and \eqref{eq:b03} are allowed to be infinite (with positive probability).
We note that square-integrability of the diffusion process $Z$
in \cref{b19}, however, is in general required; e.g., choose $A\equiv 0$
and $Z$ such that the It\^o isometry does not hold for the
It\^o integral $Y_T$.

\begin{lemma}\label{b19}Let $T\in(0,\infty)$, $d,m\in\N$,
let
$\langle \cdot,\cdot\rangle \colon \R^d\times\R^d\to\R$ denote the standard scalar product on $\R^d$,
let
$\lVert  \cdot\rVert \colon \R^d\to[0,\infty)$ denote the standard norm on $\R^d$,
let
$\lVert  \cdot\rVert_{\sfF} \colon \R^{d\times m}\to[0,\infty)$ denote the Frobenius norm on $\R^{d\times m}$,
let $(\Omega,\mathcal{F},\P, (\F_t)_{t\in[0,T]}) $ be a filtered probability space which satisfies the usual conditions,
let $W\colon [0,T]\times \Omega \to\R^m$ be a standard 
$(\F_t)_{t\in[0,T]}$-Brownian motion with continuous sample paths,
let  $Y\colon[0,T]\times\Omega\to\R^d$ be adapted
   with continuous sample paths,
let $A\colon[0,T]\times\Omega\to\R^d$ be measurable, let
$Z\colon[0,T]\times\Omega\to\R^{d\times m}$
   be progressively
   measurable,
 and assume that for all $s\in [0,T]$ 
 it holds  a.s.\ that
\begin{align}\label{a01}
\int_0^T \Bigl(\lVert A_t\rVert+\E[\lVert Z_t\rVert_{\sfF}^2]\Bigr)\,dt<\infty\quad\text{and}\quad   Y_s=Y_T+\int_s^T A_t\,dt-\int_s^TZ_t\,dW_t.
\end{align}
Then  \begin{enumerate}[(i)]
\item \label{b01}for all $s\in[0,T]$, $\lambda\in (0,\infty)$  
it holds   a.s.\ that
\begin{align}\label{eq:b01}
\E\!\left[
e^{\lambda s}\lVert  Y_s\rVert ^2+
 \int_{s}^{T}e^{\lambda t}\lVert Z_{t}\rVert_{\sfF}^2\,dt\Big|\F_s\right]
\leq \E\!\left[e^{\lambda T}\lVert  Y_T\rVert ^2
+
\int_{s}^{T}\frac{e^{\lambda t}}{\lambda}
\lVert  A_{t}\rVert ^2
\,dt\Big|\F_s \right],
\end{align}
\item  \label{b02}  for all $s\in[0,T]$, $\lambda\in (0,\infty)$  
it holds   a.s.\ that
\begin{align}\label{eq:b02}
  \E\!\left[\sup_{t\in[s,T]}\left( e^{\lambda t}\lVert Y_t\rVert^2
  +\int_t^{T} e^{\lambda u}\lVert Z_u\rVert_{\sfF}^2\,du\right)\middle| \F_s\right]
  &\leq 34\E\!\left[ e^{\lambda T}\lVert Y_{T}\rVert^2+ \int_s^{T}\frac{e^{\lambda t}}{\lambda}\lVert A_t\rVert^2\,dt\middle| \F_s\right],
\end{align}
and 
\item \label{b06} it holds for all $\alpha,\lambda\in(0,\infty)$ that
\begin{align}\begin{split}\label{eq:b03}
&
\int_0^T\left[
\frac{t^{\alpha-1}e^{\lambda t}
\E\!\left[
\left\lVert   Y_t\right\rVert^2\right]}{\Gamma(\alpha)}+
\frac{t^{\alpha} e^{\lambda t}\E\!\left[
\left\lVert   Z_t \right\rVert_{\sfF}^2\right]}{\Gamma(\alpha+1)}\right]dt
\leq 
\frac{e^{\lambda T}T^\alpha\E\!\left[\lVert  Y_T\rVert ^2\right]}{\Gamma(\alpha+1)}
+\frac{1}{\lambda}\int_0^T
\frac{e^{\lambda t}t^{\alpha}\E\!\left[
\lVert  A_{t}\rVert ^2\right] }{\Gamma(\alpha+1)}
dt.
\end{split}\end{align}
\end{enumerate}
\end{lemma}
\begin{proof}[Proof of \cref{b19}]\sloppy
 Throughout this proof for every $s\in[0,T]$ let $B_s\in\mathbb{F}_s$ satisfy
that 
 a.s.\ on $B_s$ it holds that 
$ \E\!\left[ \|Y_{T}\|^2
    +\int_s^{T}\|A_t\|^2\,dt\Big| \F_s\right]<\infty$
and  a.s.\ on $\Omega\setminus B_s$ it holds that 
$ \E\!\left[ \|Y_{T}\|^2
    +\int_s^{T}\|A_t\|^2\,dt\Big| \F_s\right]=\infty$, let 
$\{e_1,e_2,\ldots,e_m\}\subseteq \R^m $ 
be an orthonormal basis of $\R^m$, and let $\alpha,\lambda\in (0,\infty)$.
  First note that
\eqref{a01},  
Jensen's inequality, and the Burkholder-Davis-Gundy inequality (see, e.g., \cite[Lemma~7.2]{dz92}) yield that for all $s\in[0,T]$ 
it holds
a.s.\ on $B_s$  that
  \begin{equation}  \begin{split}\label{b03}
  \E\!\left[\sup_{t\in [s,T]}\lVert Y_t\rVert^2\Big|\mathbb{F}_s\right]
  &\leq 3 \left(\lVert Y_s\rVert^2+\E\!\left[\left(\int_s^T \lVert A_t\rVert\,dt\right)^2
  +\sup_{u\in [s,T]}\left\lVert \int_s^uZ_t\,dW_t\right\rVert^2\Big|\mathbb{F}_s\right]\right)
  \\&\leq
   12 \left(\lVert Y_s\rVert^2+\E\left[T\int_s^T\lVert A_t\rVert^2\,dt\Big|\mathbb{F}_s\right]
  + \E\left[\int_s^T\lVert Z_t\rVert_{\sfF}^2\,dt\Big|\mathbb{F}_s\right]
  \right)<\infty.
  \end{split}     \end{equation}
This, 
  the $L^1$-Burkholder-Davis-Gundy inequality (e.g., \cite[Theorem~1]{Ren08}),
  the Cauchy-Schwarz inequality, and H\"older's inequality
  imply that for all $s\in[0,T]$ it holds 
 a.s.\ on $B_s$ that
  \begin{equation}  \begin{split}\label{eq:WIntegralYZ}
   &\E\!\left[\sup_{u\in [s,T]}
\left\lvert \int_s^{u}e^{\lambda t}\langle Y_t,Z_t\,dW_t\rangle\right\rvert
\Big|\mathbb{F}_s
\right]
  \leq \sqrt{8}  \E\left[\left(\int_s^{T}e^{2\lambda t}
\left[\sum_{i=1}^{m}
\left\lvert\left \langle
Y_t,Z_te_i\right
\rangle
 \right\rvert^2\right]dt\right)^{\nicefrac{1}{2}}\Big|\mathbb{F}_s\right]\\
&\leq \sqrt{8}  \E\!\left[\left(\int_s^{T}e^{\lambda t}\left\lVert Z_t\right\rVert_{\sfF}^2\,dt\right)^{\!\nicefrac{1}{2}}\left(\sup_{t\in [s,T]}e^{\lambda t}\lVert Y_t\rVert^2\right)^{\!\nicefrac{1}{2}}\Big|\mathbb{F}_s\right]\\
  &\leq \sqrt{8} 
\left(\E\!\left[\int_s^{T}e^{\lambda t}\left\lVert Z_t\right\rVert_{\sfF}^2\,dt\Big|\mathbb{F}_s\right]
                  \E\!\left[\sup_{t\in [s,T]}e^{\lambda t}\lVert Y_t\rVert^2\Big|\mathbb{F}_s\right]
            \right)^{\frac{1}{2}}.
  \end{split}     \end{equation}
This, 
  \eqref{a01}, and
  \eqref{b03} 
    yield that for all $s\in[0,T]$ 
  it holds a.s.\ that
  $\big(\1_{B_s}\int_s^{u} e^{\lambda t}\langle Y_t,Z_t\,dW_t\rangle\big)_{u\in[s,T]}$
  is  a martingale with respect to $\P(\cdot|\mathbb{F}_s)$
  and
  \begin{equation}  \begin{split}\label{b04}
    \E\!\left[\1_{B_s}\int_s^T e^{\lambda t}\langle Y_t,Z_t\,dW_t\rangle
    \Big|\F_s\right]=0.
  \end{split}     \end{equation}
Next note that \eqref{a01} and
It\^o's formula show that for all $s\in[0,T]$ it holds  a.s.\ that
\begin{align}\begin{split}
&e^{\lambda T}\lVert  Y_T\rVert ^2
-e^{\lambda s}\lVert  Y_s\rVert ^2=
 \int_{s}^{T}d(e^{\lambda t}\lVert  Y_{t}\rVert ^2)=  \int_{s}^{T}\lambda e^{\lambda t}\lVert  Y_{t}\rVert ^2\,dt+ \int_{s}^{T}e^{\lambda t}d (\lVert  Y_{t}\rVert ^2)\\
&
= \int_{s}^{T}\lambda e^{\lambda t}\lVert  Y_{t}\rVert ^2\,dt
+
\int_{s}^{T} e^{\lambda t}\left(
\lVert Z_{t}\rVert_{\sfF}^2
-2\langle Y_{t}, A_{t}\rangle\right)dt+\int_{s}^{T}2e^{\lambda t}\langle Y_{t}, Z_{t} dW_{t}\rangle\\
&= \int_{s}^{T}e^{\lambda t}\lVert Z_{t}\rVert_{\sfF}^2\,dt
+ \int_{s}^{T}\tfrac{e^{\lambda t}}{\lambda}
\bigl[
\lVert \lambda Y_{t}-A_{t}\rVert^2 -\lVert  A_{t}\rVert ^2\bigr]\,dt+\int_{s}^{T}2e^{\lambda t}\langle Y_{t} , Z_{t} dW_{t}\rangle.
\end{split}\end{align}
This shows that for all $s\in[0,T]$ it holds a.s.\ that
\begin{align}\begin{split}
&
e^{\lambda s}\lVert  Y_s\rVert ^2+
 \int_{s}^{T}e^{\lambda t}\lVert Z_{t}\rVert_{\sfF}^2\,dt+
\int_{s}^{T}\tfrac{e^{\lambda t}}{\lambda}
\lVert \lambda Y_{t}-A_{t}\rVert^2 \,dt
\\
&=e^{\lambda T}\lVert  Y_T\rVert ^2
+
\int_{s}^{T}\tfrac{e^{\lambda t}}{\lambda}
\lVert  A_{t}\rVert ^2
\,dt-\int_{s}^{T}2e^{\lambda t}\langle Y_{t} , Z_{t} dW_{t}\rangle.
\end{split}\label{b05}\end{align}
This and \eqref{b04} show that for all $s\in[0,T]$ it holds  a.s.\ 
on $B_s$ that
\begin{align}
\E\!\left[
e^{\lambda s}\lVert  Y_s\rVert ^2+
 \int_{s}^{T}e^{\lambda t}\lVert Z_{t}\rVert_{\sfF}^2\,dt\Big|\F_s\right]
\leq \E\!\left[e^{\lambda T}\lVert  Y_T\rVert ^2
+
\int_{s}^{T}\tfrac{e^{\lambda t}}{\lambda}
\lVert  A_{t}\rVert ^2
\,dt\Big|\F_s \right].\label{b01a}
\end{align} This and the definition of $B_s$, $s\in [0,T]$, prove \eqref{b01}.

Next observe that
\eqref{eq:WIntegralYZ}, \eqref{b04} and \eqref{b01}
  yield that for all $s\in[0,T]$ it holds    a.s.\ on  $B_s$ that
  \begin{equation}
  \begin{split}
  &\E\bigg[\sup_{t\in [s,T]}\left(-2\int_t^{ T} e^{\lambda u}\langle Y_u,Z_udW_u\rangle\right)\Big|\mathbb{F}_s\bigg]
  \\&
  =
  2\E\bigg[\sup_{t\in [s,T]}\left(\int_s^{ t} e^{\lambda u}\langle Y_u,Z_u\,dW_u\rangle\right)
  \big|\mathbb{F}_s\bigg]
  -2\E\bigg[\int_s^{T} e^{\lambda u}\langle Y_u,Z_u\,dW_u\rangle
  \big|\mathbb{F}_s\bigg]
  \\
  &\leq 2\sqrt{8} \left(
     \E\!\left[\int_s^{T} e^{\lambda t}\left\lVert Z_t\right\rVert^2_{\sfF}\,dt
                    \big|\mathbb{F}_s\right]
  \E\!\left[\sup_{t\in [s,T]} e^{\lambda t}\|Y_t\|^2\big|\mathbb{F}_s\right]
            \right)^{\frac{1}{2}}.
  \\
  &\leq 2\sqrt{8} \left(
     \E\!\left[e^{\lambda T}\|Y_{T}\|^2+\int_s^{T}\tfrac{e^{\lambda t}}{\lambda}\|A_t\|^2\,dt
                    \big|\mathbb{F}_s\right]
  \E\!\left[\sup_{t\in [s,T]} e^{\lambda t}\|Y_t\|^2\big|\mathbb{F}_s\right]
            \right)^{\frac{1}{2}}.
  \end{split}
  \end{equation}
This, \eqref{b05}, \eqref{a01}, and \eqref{b03}
  yield that for all $s\in[0,T]$  it holds   a.s.\ on  $B_s$ that
  \begin{equation} \begin{split}
    &\E\!\left[
    \sup_{t\in [s,T]}\left( e^{\lambda t}\lVert Y_t\rVert^2+
    \int_t^{T} e^{\lambda u}\lVert Z_u\rVert^2_{\sfF}\,du\right)
    \big|\mathbb{F}_s\right]
  \\
  &\leq\E\!\left[ e^{\lambda T}\lVert Y_{T}\rVert^2+\int_s^{T}\frac{e^{\lambda t}}{\lambda}\lVert A_t\rVert^2\,dt
    \big|\mathbb{F}_s\right]
    +\E\left[\sup_{t\in [s,T]}\left(-2\int_t^{T} e^{\lambda u}\langle Y_u,Z_u\,dW_u\rangle\right)
    \big|\mathbb{F}_s\right]
  \\
  &\leq\E\!\left[ e^{\lambda T}\lVert Y_{T}\rVert^2+\int_s^{T}\frac{e^{\lambda t}}{\lambda}\lVert A_t\rVert^2\,dt
    \big|\mathbb{F}_s\right]
    +2\sqrt{8}\left(
    \E\left[ e^{\lambda T}\lVert Y_{T}\rVert^2+\int_s^{T}\frac{e^{\lambda t}}{\lambda}\lVert A_t\rVert^2\,dt
    \big|\mathbb{F}_s\right]
              \right)^{\frac{1}{2}}
    \\&\qquad
    \cdot
      \left(\E\!\left[\sup_{t\in [s,T]}\left( e^{\lambda t}\lVert Y_t\rVert^2
                          +\int_t^{T} e^{\lambda u}\lVert Z_u\rVert^2_{\sfF}\,du\right)
     \big|\F_s\right]\right)^{\frac{1}{2}}<\infty.
     \label{eq:final.esti}
  \end{split} \end{equation}
This and the fact that
$\forall\,x,c\in[0,\infty)\colon \bigl(\bigl[x\leq c+2\sqrt{8}\sqrt{c}\sqrt{x}\bigr]\Rightarrow\bigl[ x\leq \bigl(\sqrt{8}+\sqrt{8+1}\bigr)^2 c\leq 34c\bigr]\bigr)$ imply  for all $s\in[0,T]$ that  a.s.\ on  $B_s$ it holds that
\begin{align}
  \E\!\left[\sup_{t\in[s,T]}\left( e^{\lambda t}\lVert Y_t\rVert^2
  +\int_t^{T} e^{\lambda u}\lVert Z_u\rVert_{\sfF}^2\,du\right)\middle| \F_s\right]
  &\leq 34\E\!\left[ e^{\lambda T}\lVert Y_{T}\rVert^2+ \int_s^{T}\frac{e^{\lambda t}}{\lambda}\lVert A_t\rVert^2\,dt\middle| \F_s\right].
\end{align}
This and the definition of $B_s$, $s\in [0,T]$, prove \eqref{b02}.

Next,  the fact that
$\forall\, t\in [0,T]\colon \frac{t^{\alpha}}{\Gamma(\alpha+1)}= 
\int_{0}^{t}\frac{s^{\alpha-1}\,ds}{\Gamma(\alpha)}
$, Tonelli's theorem, the tower property, and
\eqref{b01} show
 that
\begin{align}
&
\int_0^T\left[
\frac{t^{\alpha-1}}{\Gamma(\alpha)}e^{\lambda t}
\E\!\left[
\left\lVert   Y_t\right\rVert^2\right]+
\frac{t^{\alpha} e^{\lambda t}}{\Gamma(\alpha+1)}
\E\!\left[
\left\lVert   Z_t \right\rVert_{\sfF}^2\right]\right]\!dt\nonumber  \\
&
= \int_0^T
\frac{t^{\alpha-1}}{\Gamma(\alpha)}e^{\lambda t}
\E\!\left[
\left\lVert   Y_t\right\rVert^2\right]\!dt+
\int_0^T\int_{0}^{t}\frac{s^{\alpha-1} e^{\lambda t}}{\Gamma(\alpha)}
\E\!\left[
\left\lVert   Z_t\right\rVert_{\sfF}^2\right]\!dsdt  \nonumber \\
&
= \int_0^T
\frac{s^{\alpha-1}}{\Gamma(\alpha)}
e^{\lambda s}
\E\!\left[
\left\lVert   Y_s\right\rVert^2\right]\!ds
+
\int_0^T\int_{s}^{T}\frac{s^{\alpha-1} e^{\lambda t}}{\Gamma(\alpha)}
\E\!\left[
\left\lVert   Z_t\right\rVert_{\sfF}^2\right]\!dtds \nonumber  \\
&
=
\int_0^T\frac{s^{\alpha-1} }{\Gamma(\alpha)}
\E\!\left[\E\!\left[e^{\lambda s} \left\lVert Y_s\right\rVert^2+
\int_{s}^{T}
e^{\lambda t}
\left\lVert   Z_t\right\rVert_{\sfF}^2 dt\Bigr|\F_s\right]\right]\!ds \nonumber  \\
&\leq \int_0^T\frac{s^{\alpha-1} }{\Gamma(\alpha)}
\E\!\left[
\E\!\left[e^{\lambda T}\lVert  Y_T\rVert ^2
+
\int_{s}^{T}\frac{e^{\lambda t}}{\lambda}
\lVert  A_{t}\rVert ^2
\,dt\Big|\F_s \right]\right]\!ds \nonumber  \\
&= e^{\lambda T}
\E\!\left[\lVert  Y_T\rVert ^2\right]
\left(\int_0^T\frac{s^{\alpha-1}\,ds }{\Gamma(\alpha)}\right)
+
 \int_0^T\int_{0}^{t}\frac{s^{\alpha-1} }{\Gamma(\alpha)}
\frac{e^{\lambda t}}{\lambda}
\E\!\left[
\lVert  A_{t}\rVert ^2\right]dsdt\nonumber  \\
&=\frac{e^{\lambda T}T^\alpha\E\!\left[\lVert  Y_T\rVert ^2\right]}{\Gamma(\alpha+1)}
+\frac{1}{\lambda}\int_0^T
\frac{e^{\lambda t}t^{\alpha}\E\!\left[
\lVert  A_{t}\rVert ^2\right] }{\Gamma(\alpha+1)}\,
dt.
\end{align}
This shows \eqref{b06}. The proof of \cref{b19} is thus completed.
\end{proof}
\section{Upper bounds for the convergence speed of Picard iterations}\label{sec:upper_bound}
In this section we provide upper bounds for the convergence speed of Picard iterations of BSDEs. \cref{b20} establishes an explicit bound for the $L^2$-distance between the Picard iterations and the solution of a BSDE with a globally Lipschitz continuous nonlinearity. Our proof of \cref{b20} relies on the a priori estimates for backward It\^o processes provided in \cref{b19}. In \cref{r01} we employ the estimate of \cref{b20} to obtain the square root-factorial speed of convergence of Picard iterations. In \cref{r02} we employ the estimate of \cref{b20} to obtain the factorial speed of convergence of Picard iterations in the z-independent case. 


\begin{proposition}\label{b20}
Let
$T\in (0,\infty)$,
 $d,m\in\N$, $\lipconstY,\lipconstZ \in [0,\infty) $,
let $0^0=1$,
let
$\langle \cdot,\cdot\rangle \colon \R^d\times\R^d\to\R$ denote the standard scalar product on $\R^d$,
let
$\lVert  \cdot\rVert \colon \R^d\to[0,\infty)$ denote the standard norm on $\R^d$,
let
$\lVert  \cdot\rVert_{\sfF} \colon \R^{d\times m}\to[0,\infty)$ denote the Frobenius norm on $\R^{d\times m}$,
let $(\Omega,\mathcal{F},\P, (\F_t)_{t\in[0,T]}) $ be a filtered probability space which satisfies the usual conditions, 
let $f\colon[0,T]\times \Omega \times \R^d\times\R^{d\times m}\to\R^d$ be measurable, 
assume that for all $t\in [0,T]$, $y,\tilde{y}\in\R^d$, $z,\tilde{z}\in\R^{d\times m}$ it holds a.s.\
  that
\begin{equation}  
\begin{split}
 \left   \lVert f(t,y,z)-f(t,\tilde{y},\tilde{z})\right\rVert
    \leq \lipconstY \lVert y-\tilde{y}\rVert
    + \lipconstZ \lVert z-\tilde{z}\rVert_{\sfF},
\end{split}  \label{c03}   
\end{equation}
let $W\colon [0,T]\times \Omega \to\R^m$ be a standard 
$(\F_t)_{t\in[0,T]}$-Brownian motion with continuous sample paths,
let $\xi \colon \Omega\to\R^d$ be $\F_T$-measurable,
let
$Y^k\colon[0,T]\times\Omega\to\R^d$, 
$k\in \N_0\cup\{\infty\}$,
be adapted
 with continuous sample paths,
let 
$Z^k\colon[0,T]\times\Omega\to\R^{d\times m}$,
$k\in \N_0\cup\{\infty\}$,
be progressively measurable, and
assume that
for all $s\in [0,T]$, $k\in\N_0\cup\{ \infty \}$
it holds  a.s.\  that
 $\int_0^T\E[\lVert\xi\rVert^2+\lVert f(t,0,0)\rVert^2+\lVert Y^\infty_t\rVert^2+\lVert Z_t^k\rVert^2_{\sfF}]\,dt<\infty$,
$Y^0_s=0$, $Z^0_s=0$,
and
\begin{align} \label{b20a}
Y^{k+1}_{s}=\xi+\int_{s}^{T}f(t,Y_t^k, Z_t^k)\,dt- \int_{s}^{T}
Z^{k+1}_t\,dW_t.
\end{align}
Then  it holds 
for all $k\in\N$
that
 \begin{align}\begin{split}
&
\E\!\left[\sup_{t\in[0,T]}\left(\left\lVert Y^{k}_t-Y_t^\infty\right\rVert^2\right)+\int_0^{T}
\left\lVert Z^{k}_t-Z_t^\infty\right\rVert^2_{\sfF}\,dt\right]\\
&
\leq 
35
\left(\frac{Te}{k}\right)^{k}
\left[
 \sum_{\ell=0}^{k}
\frac{k!\lipconstY ^{\ell}\lipconstZ ^{k-\ell}T^{\ell/2}}{\ell! (k-\ell)!\sqrt{\ell!}}
\right]^2
\left(
\E\bigl[\lVert \xi\rVert^2\bigr]+\frac{T}{k}
\int_0^T
\E\!\left[
\bigl\lVert f(t,Y_t^\infty,Z_t^\infty)\bigr\rVert^2\right]\!dt\right)<\infty.
\end{split}\end{align}
\end{proposition}

\begin{proof}[Proof of \cref{b20}]
First note that  \eqref{b20a} proves that
for all $s\in [0,T]$, $k\in\N_0$ it holds a.s.\ that
\begin{align}\begin{split}
&
 Y^{k}_s-Y_s^\infty\\
&=-\1_{\{0\}}(k)\xi+
\int_{s}^{T}\Bigl[ \1_{\N}(k)
f(s,Y^{\lvert k-1\rvert}_s,Z^{\lvert k-1\rvert}_s)-f(s,Y_s^\infty,Z_s^\infty)\Bigr]dt
-\int_{s}^{T}\bigl[
  Z^{k}_t-Z_t^\infty\bigr]dW_t.
\end{split}\label{k01}\end{align}
This, the tower property, 
Tonelli's theorem,
and 
\cref{b19} (applied for every $k\in \N_0$ with $Y\gets Y^{k} -Y ^\infty$, $A\gets \1_{\N}(k)
f(\cdot,Y^{\lvert k-1\rvert} ,Z^{\lvert k-1\rvert} )-f(\cdot,Y ^\infty,Z ^\infty)$,
$Z\gets  Z^{k} -Z ^\infty$ in the notation of \cref{b19}) 
 prove

\begin{enumerate}[(i)]
\item  that for all $k\in\N_0$, $\lambda \in (0,\infty)$ it holds that
\begin{equation}\begin{split}
&
  \E\!\left[\left[\sup_{t\in[0,T]}\left( e^{\lambda t}\lVert  Y^{k}_t-Y_t^\infty\rVert^2\right)\right]
  +\int_0^{T} e^{\lambda t}\left\lVert   Z^{k}_t-Z_t^\infty\right\rVert^2_{\sfF}\,dt\right]\\
  &\leq \frac{35}{\lambda}\E\!\left[ \lambda e^{\lambda T}\lVert \xi\rVert^2\1_{\{0\}}(k)+\int_0^{T}e^{\lambda t}\left\lVert \1_{\N}(k)
f(t,Y^{\lvert k-1\rvert}_t,Z^{\lvert k-1\rvert}_t)-f(t,Y_t^\infty,Z_t^\infty)\right\rVert^2dt\right],\label{k05}
\end{split}
\end{equation}
and
\item that
for all
$k\in\N_0$,  $\alpha,\lambda\in (0,\infty)$  it holds that
\begin{align}\begin{split}
&
\int_0^T
\frac{t^{\alpha-1}e^{\lambda t}\E\!\left[
\left\lVert    Y^{k}_t-Y_t^\infty\right\rVert^2\right]}{\Gamma(\alpha)}
+
\frac{t^{\alpha} e^{\lambda t}\E\!\left[
\left\lVert    Z^{k}_t-Z_t^\infty \right\rVert_{\sfF}^2\right]}{\Gamma(\alpha+1)}
dt\\
&
\leq 
\frac{e^{\lambda T}T^\alpha
\E\!\left[\lVert \xi\rVert ^2\right] \1_{\{0\}}(k)  }{\Gamma(\alpha+1)}
+\frac{1}{\lambda}\int_0^T
\frac{e^{\lambda t}t^{\alpha}\E\!\left[
\left
\lVert \1_{\N}(k)
f(t,Y^{\lvert k-1\rvert}_t,Z^{\lvert k-1\rvert}_t)-f(t,Y_t^\infty,Z_t^\infty)\right\rVert ^2\right] }{\Gamma(\alpha+1)}\,
dt.
\end{split}\end{align}
\end{enumerate}
This, \eqref{c03}, and the fact that
$\forall\,\alpha\in\N_0\colon \Gamma(\alpha+1)=\alpha! $
show for all 
$\alpha,k\in\N_0$, $\lambda\in (0,\infty)$ that 
\begin{align}
&
\left(\int_0^T\frac{t^{\alpha}e^{\lambda t}}{\alpha!}
\E\!\left[
\bigl\lVert f(t,Y^{k}_t,Z^{k}_t)-f(t, Y_t^\infty, Z^\infty_t)\bigr\rVert^2\right]\!dt\right)^{\!\nicefrac{1}{2}}\nonumber \\
&\leq 
\lipconstY 
\left(\int_0^T\frac{t^{\alpha}e^{\lambda t}}{\alpha!}
\E\!\left[
\bigl\lVert Y^{k}_t- Y_t^\infty\bigr\rVert^2\right]\!dt\right)^{\!\nicefrac{1}{2}}
+\lipconstZ 
\left(\int_0^T\frac{t^{\alpha}e^{\lambda t}}{\alpha!}
\E\!\left[
\bigl\lVert Z^{k}_t- Z_t^\infty\bigr\rVert^2_{\sfF}\right]\!dt\right)^{\!\nicefrac{1}{2}}\nonumber 
\\
&\leq \sum_{\nu=0}^{1}\Biggl[\frac{\lipconstY ^{\nu}\lipconstZ ^{1-\nu}}{\sqrt{\lambda}}
\Biggl(\frac{T^{\alpha+\nu}}{(\alpha+\nu)!}\lambda
e^{\lambda T}\E\bigl[\lVert \xi\rVert^2\bigr]\1_{\{0\}}(k)\nonumber \\
&\qquad\qquad+
\int_0^T
\frac{ t^{\alpha+\nu} e^{\lambda t}}{(\alpha+\nu)!}
\E\!\left[
\bigl\lVert \1_{\N}(k)
f(t,Y^{\lvert k-1\rvert}_t,Z^{\lvert k-1\rvert}_t)-f(t,Y_t^\infty,Z_t^\infty)\bigr\rVert^2\right]\!dt\Biggr)^{\!\nicefrac{1}{2}}
\Biggr].\label{a15}
\end{align}
This and induction prove for all $k\in \N\cap[2,\infty)$, $\lambda\in (0,\infty)$ that
\begin{align}
&
 \left(
\int_0^T\frac{t^{0}e^{\lambda t}}{0!}
\E\!\left[
\bigl\lVert f(t,Y^{ k-1}_t,Z^{ k-1}_t)-f(t,Y_t^\infty,Z_t^\infty)\bigr\rVert^2\right]\!dt\right)^{\!\nicefrac{1}{2}}
\nonumber \\
&\leq 
\sum_{\nu_1,\nu_2,\ldots,\nu_{k-1}=0}^{1}\left[
\frac{\lipconstY ^{\sum_{i=1}^{k-1}\nu_i}\lipconstZ ^{k-1-\sum_{i=1}^{k-1}\nu_i}}{\lambda^{(k-1)/2}}
\left(
\int_0^T
\frac{ t^{\sum_{i=1}^{k-1}\nu_i} e^{\lambda t}}{(\sum_{i=1}^{k-1}\nu_i)!}
\E\!\left[
\bigl\lVert 
f(t,Y_t^0,Z_t^0)-f(t,Y_t^\infty,Z_t^\infty)
\bigr\rVert^2\right]\!dt\right)^{\!\nicefrac{1}{2}}
\right]
\nonumber \\
&\leq 
\sum_{\nu_1,\nu_2,\ldots,\nu_{k-1}=0}^{1}\Biggl[
\frac{\lipconstY ^{\sum_{i=1}^{k-1}\nu_i}\lipconstZ ^{k-1-\sum_{i=1}^{k-1}\nu_i}}{\lambda^{(k-1)/2}}\sum_{\nu_k=0}^{1}\frac{\lipconstY ^{\nu_k}\lipconstZ ^{1-\nu_k}}{\sqrt{\lambda}}\nonumber \\
&\qquad\qquad \qquad\qquad\cdot 
\left(
\frac{T^{\sum_{i=1}^{k}\nu_i}}{(\sum_{i=1}^{k}\nu_i)!}\lambda
e^{\lambda T}\E\bigl[\lVert \xi\rVert^2\bigr]+
\int_0^T
\frac{ t^{\sum_{i=1}^{k}\nu_i} e^{\lambda t}}{(\sum_{i=1}^{k}\nu_i)!}
\E\!\left[
\bigl\lVert f(t,Y_t^\infty,Z_t^\infty)\bigr\rVert^2\right]\!dt\right)^{\!\nicefrac{1}{2}}
\Biggr]\nonumber \\
&=\sum_{\nu_1,\nu_2,\ldots,\nu_{k}=0}^{1}
\frac{\lipconstY ^{\sum_{i=1}^{k}\nu_i}\lipconstZ ^{k-\sum_{i=1}^{k}\nu_i}}{\lambda^{k/2}}
\nonumber \\&\qquad\qquad\qquad\cdot \left(
\frac{T^{\sum_{i=1}^{k}\nu_i}}{(\sum_{i=1}^{k}\nu_i)!}\lambda
e^{\lambda T}\E\bigl[\lVert \xi\rVert^2\bigr]+
\int_0^T
\frac{ t^{\sum_{i=1}^{k}\nu_i} e^{\lambda t}}{(\sum_{i=1}^{k}\nu_i)!}
\E\!\left[
\bigl\lVert  f(t,Y_t^\infty,Z_t^\infty)\bigr\rVert^2\right]\!dt\right)^{\!\nicefrac{1}{2}}
\label{c02}
\end{align}
This and \eqref{a15} show for all $k\in \N$, $\lambda\in (0,\infty)$ that
\begin{align}\begin{split}
&
 \left(
\int_0^T e^{\lambda t}
\E\!\left[
\bigl\lVert f(t,Y^{ k-1}_t,Z^{ k-1}_t)-f(t,Y_t^\infty,Z_t^\infty)\bigr\rVert^2\right]\!dt\right)^{\!\nicefrac{1}{2}}
\\
&\leq \sum_{\ell=0}^{k}\left[
\frac{k!}{\ell! (k-\ell)!}
\frac{\lipconstY ^{\ell}\lipconstZ ^{k-\ell}}{\lambda^{k/2}}
\left(
\frac{T^{\ell}}{\ell!}\lambda
e^{\lambda T}\E\bigl[\lVert \xi\rVert^2\bigr]+
\int_0^T
\frac{ t^{\ell} e^{\lambda t}}{\ell!}
\E\!\left[
\bigl\lVert  f(t,Y_t^\infty,Z_t^\infty)\bigr\rVert^2\right]\!dt\right)^{\!\nicefrac{1}{2}}
\right]
\\
&\leq\left[ \sum_{\ell=0}^{k}
\frac{k!}{\ell! (k-\ell)!}
\frac{\lipconstY ^{\ell}\lipconstZ ^{k-\ell}}{\lambda^{k/2}}
\frac{T^{\ell/2}e^{\lambda T/2}}{\sqrt{\ell!}}
\right]
\left(
\lambda
\E\bigl[\lVert \xi\rVert^2\bigr]+
\int_0^T
\E\!\left[
\bigl\lVert f(t,Y_t^\infty,Z_t^\infty)\bigr\rVert^2\right]\!dt\right)^{\!\nicefrac{1}{2}}.
\end{split}\label{c04}\end{align}
This and \eqref{k05} prove for all $k\in \N$, $\lambda\in (0,\infty)$ 
that 
\begin{align}\begin{split}
&\E\!\left[\sup_{t\in[0,T]}\left(e^{\lambda t}\left\lVert Y^{k}_t-Y_t^\infty\right\rVert^2\right)+\int_0^{T}e^{\lambda t}\left\lVert Z^{k}_t-Z_t^\infty\right\rVert^2_{\sfF}\,dt\right]\\
&
\leq
\frac{35}{\lambda} \E\!\left[\int_0^{T}e^{\lambda t} 
\left\lVert
f(t,Y^{ k-1}_t,Z^{ k-1}_t)-f(t,Y_t^\infty,Z_t^\infty)\right\rVert^2
\,dt\right] 
\\
&\leq
 \frac{35}{\lambda}\left[
 \sum_{\ell=0}^{k}
\frac{k!}{\ell! (k-\ell)!}
\frac{\lipconstY ^{\ell}\lipconstZ ^{k-\ell}}{\lambda^{k/2}}
\frac{T^{\ell/2}e^{\lambda T/2}}{\sqrt{\ell!}}
\right]^2
\left(
\lambda
\E\bigl[\lVert \xi\rVert^2\bigr]+
\int_0^T
\E\!\left[
\bigl\lVert f(t,Y_t^\infty,Z_t^\infty)\bigr\rVert^2\right]\!dt\right) .
\label{k07}\end{split}
\end{align}
Furthermore, observe  for all
$k\in\N$  that
\begin{align}
&
\left[
 \sum_{\ell=0}^{k}
\frac{k!}{\ell! (k-\ell)!}
\frac{\lipconstY ^{\ell}\lipconstZ ^{k-\ell}}{\lambda^{k/2}}
\frac{T^{\ell/2}e^{\lambda T/2}}{\sqrt{\ell!}}
\right]^2\Biggr|_{\lambda=\frac{k}{T}}
=
\left(\frac{Te}{k}\right)^{k}
\left[
 \sum_{\ell=0}^{k}
\frac{k!\lipconstY ^{\ell}\lipconstZ ^{k-\ell}T^{\ell/2}}{\ell! (k-\ell)!\sqrt{\ell!}}
\right]^2.
\end{align}
This and \eqref{k07}
yield for all
$k\in\N$ that
\begin{align}\begin{split}
&
\E\!\left[\sup_{t\in[0,T]}\left(\left\lVert Y^{k}_t-Y_t^\infty\right\rVert^2\right)+\int_0^{T}
\left\lVert Z^{k}_t-Z_t^\infty\right\rVert^2_{\sfF}\,dt\right]\\
&
\leq 
\frac{35T}{k}
\left(\frac{Te}{k}\right)^{k}
\left[
 \sum_{\ell=0}^{k}
\frac{k!\lipconstY ^{\ell}\lipconstZ ^{k-\ell}T^{\ell/2}}{\ell! (k-\ell)!\sqrt{\ell!}}
\frac{}{}
\frac{}{}
\right]^2
\left(
\frac{k}{T}
\E\bigl[\lVert \xi\rVert^2\bigr]+
\int_0^T
\E\!\left[
\bigl\lVert f(t,Y_t^\infty,Z_t^\infty)\bigr\rVert^2\right]\!dt\right).
\end{split}\label{k08}\end{align}
Next note that \eqref{c03} ensures that
\begin{align}
&\left(\int_0^T
\E\bigl[
\lVert f(t,Y_t^\infty,Z_t^\infty)\rVert^2\bigr]dt\right)^{\nicefrac{1}{2}}
\\
&\leq 
\left(\int_0^T
\E\bigl[
\lVert f(t,0,0)\rVert^2\bigr]dt\right)^{\nicefrac{1}{2}}
+
\lipconstY
\left(
\int_0^T
\E\bigl[
\lVert Y_t^\infty \rVert^2\bigr]dt\right)^{\nicefrac{1}{2}}
+
\lipconstZ
\left(
\int_0^T
\E\bigl[
\lVert Z_t^\infty\rVert^2_{\sfF}\bigr]dt\right)^{\nicefrac{1}{2}}<\infty.\nonumber
\end{align}
This, the fact that
$\E [\lVert\xi\rVert^2]<\infty$, and
\eqref{k08}
complete the proof of \cref{b20}.
\end{proof}
\begin{remark}\label{r01}Assume the setting of \cref{b20}. Then
it holds for all $k\in\N$ that
\begin{align}\begin{split}
&
35
\left(\frac{Te}{k}\right)^{k}
\left[
 \sum_{\ell=0}^{k}
\frac{k!\lipconstY ^{\ell}\lipconstZ ^{k-\ell}T^{\ell/2}}{\ell! (k-\ell)!\sqrt{\ell!}}
\right]^2\leq 
35\left(\frac{\max\{T^2,1\}e\max\{\lipconstY^2,\lipconstZ^2 \}}{k}\right)^{k}
\left[ \sum_{\ell=0}^{k}\frac{k!}{(k-\ell)!\ell!}\right]^2\\
&=35\left(\frac{4\max\{T^2,1\}e\max\{\lipconstY^2,\lipconstZ^2 \}}{k}\right)^{k}
\leq 
35\frac{\left(4\max\{T^2,1\}e\max\{\lipconstY^2,\lipconstZ^2 \}\right)^{k}}{k!}
.
\end{split}\end{align}
\end{remark}

\begin{remark}\label{r02}Assume the setting of \cref{b20} and assume that $\lipconstZ=0 $.  Then
it holds for all $k\in\N$ that
\begin{align}\begin{split}
&
35
\left(\frac{Te}{k}\right)^{k}
\left[
 \sum_{\ell=0}^{k}
\frac{k!\lipconstY ^{\ell}\lipconstZ ^{k-\ell}T^{\ell/2}}{\ell! (k-\ell)!\sqrt{\ell!}}
\right]^2=35
\left(\frac{Te}{k}\right)^{k}
\left[
\frac{\lipconstY ^{k}T^{k/2}}{ \sqrt{k!}}
\right]^2\leq 35 \frac{(T^2e\lipconstY^2)^k}{(k!)^2}.
\end{split}\end{align}
\end{remark}

\subsubsection*{Acknowledgements}
This work has been funded by the Deutsche Forschungsgemeinschaft (DFG, German Research Foundation) through the research grant 
 HU1889/7-1.

{
\bibliographystyle{acm}
\bibliography{bibfile}

}

\end{document}